\documentclass[a4paper, draft]{article}
\usepackage{enumerate}
\usepackage{amsmath}
\usepackage{amsfonts}
\usepackage{amssymb}
\usepackage{amsthm}
\usepackage{tikz}
\usepackage{tikz-cd}
\DeclareMathOperator{\id}{id}
\newcommand{\bdy}{\ensuremath{\partial}}
\newcommand{\sph}[1]{\ensuremath{\mathbb{S}^{#1}}}

\newcommand{\iso}{\ensuremath{\cong}}
\newcommand{\Z}[1][]{\ensuremath{\mathbb{Z}_{#1}}}

\newcommand{\R}{\ensuremath{\mathbb{R}}}
\newcommand{\I}{\ensuremath{\mathbb{I}}}

\newcommand{\E}{\ensuremath{\mathbb{E}}}
\newcommand{\N}{\ensuremath{\mathbb{N}}}

\newcommand{\HR}{\ensuremath{\mathbb{H}^2\times\R}}
\newcommand{\SLR}{\ensuremath{\widetilde{\mathrm{SL}_2(\R)}}}
\newcommand{\proP}[2][p]{\ensuremath{\widehat{#2}_{(#1)}}}
\newcommand{\nsgp}[1][]{\ensuremath{\triangleleft_{\rm #1}}}
\newcommand{\sbgp}[1][]{\ensuremath{\leq_{\rm #1}}}
\newcommand{\ofg}[1][]{\ensuremath{\pi_1^{\text{orb}}#1}}

\newcommand{\gp}[1]{\ensuremath{\langle #1\rangle}}
\newtheorem{theorem}{Theorem}[section]
\newtheorem*{thmquote}{Theorem}
\newtheorem{prop}[theorem]{Proposition}

\newtheorem{lem}[theorem]{Lemma}
\newtheorem{clly}[theorem]{Corollary}

\theoremstyle{definition}
\newtheorem{defn}[theorem]{Definition}
\newtheorem*{cnv}{Conventions}
\theoremstyle{remark}
\newtheorem*{rmk}{Remark}
\newtheorem{example}[theorem]{Example}
\newcounter{introthmcount}
\newenvironment{introthm}[1]{\\[1.9ex] {\bf Theorem\refstepcounter{introthmcount} \label{#1}\Alph{introthmcount}.} \em}{\em \\[2ex]}
\newcommand{\SFS}{Seifert fibre space}
\newcommand{\SSG}[1][X]{\ensuremath{#1_{\rm SF}}}
\theoremstyle{definition}
\newtheorem*{defnrecall}{Definition \ref{exauto}}
\title{Profinite rigidity of graph manifolds and JSJ decompositions of 3-manifolds}
\author{Gareth Wilkes}
\begin{document}
\maketitle
\begin{abstract}
There has been much recent interest into those properties of a 3-manifold determined by the profinite completion of its fundamental group. In this paper we give readily computable criteria specifying precisely when two orientable graph manifold groups have isomorphic profinite completions. Our results also distinguish graph manifolds among the class of all 3-manifolds and give information about the structure of totally hyperbolic manifolds, and give control over the pro-$p$ completion of certain graph manifold groups.
\end{abstract}
\section{Introduction}
Much attention has been paid recently to those properties of 3-manifolds which can be deduced from the finite quotients of their fundamental groups; or, from another viewpoint, from the structure of their lattice of finite-sheeted coverings. Having assembled these finite quotients into the profinite completion of the fundamental group, this amounts to the study of `profinite invariants' of the 3-manifold. A profinite invariant may be defined as some property $\cal P$ of a group $G$ such that, whenever $H$ is a group with property $\cal P$ and $\hat G\iso\hat H$, then $G$ also has property $\cal P$. One may restrict attention to a particular class of groups (for example fundamental groups of compact orientable 3-manifolds) and require both $G$ and $H$ to be from that class. For the rest of the paper we restrict our attention to this class of groups. By abuse of language one also refers to a profinite invariant of a 3-manifold.

The strongest profinite invariant one could hope for is the isomorphism type of the group itself. In this case one refers to a group whose profinite completion determines the isomorphism type of the group (again, among compact orientable 3-manifold groups) as `profinitely rigid'. The author showed recently \cite{Wilkes15} that, apart from some limited examples due to Hempel \cite{hempel14}, closed Seifert fibre spaces are profinitely rigid. Bridson and Reid \cite{BR15} and Boileau and Friedl \cite{BF15} have shown that the figure-eight knot complement, along with a handful of other knot complements, are profinitely rigid. Bridson, Reid and Wilton \cite{BRW16} have also proved that the fundamental groups of once-punctured torus bundles are profinitely rigid. To date these are the only known examples of profinitely rigid manifolds in the literature.

An important property of a 3-manifold is its geometry: whether it admits one of the eight Thurston geometries and, if so, which one. This problem is more naturally stated only for closed manifolds, as then any geometry is unique. Wilton and Zalesskii \cite{WZ14} have shown that the geometry of a 3-manifold is a profinite invariant. The aspect of this theorem distinguishing  between Seifert fibred or hyperbolic manifolds on the one hand and Sol or non-geometric manifolds on the other may be interpreted as the statement `triviality of the JSJ decomposition is a profinite invariant'. 

In this paper we investigate graph manifolds and, more generally, the JSJ decomposition. We obtain in particular criteria determining precisely when two graph manifold groups can have isomorphic profinite completion. The exact statement requires a good deal of terminology and notation which is inappropriate in an Introduction, hence we will give a rather loose paraphrasing of the result. For the precise statement see Theorem \ref{GMrigid}.
\begin{introthm}{introGM}
Let $M$, $N$ be closed orientable graph manifolds with JSJ decompositions $(X,M_\bullet)$, $(Y,N_\bullet)$ respectively. Suppose $\widehat{\pi_1 M}\iso \widehat{\pi_1 N}$. Then the graphs $X,Y$ are isomorphic, such that corresponding vertex groups have isomorphic profinite completions. Furthermore:
\begin{itemize}
\item If $X$ is not a bipartite graph, then $M$ is profinitely rigid and so $\pi_1 M\iso \pi_1 N$.
\item If $X$ is a bipartite graph then there is an explicit finite list of numerical equations defined in terms of $M,N$ which admit a solution if and only if $\widehat{\pi_1 M}\iso \widehat{\pi_1 N}$.
\end{itemize}
 In particular, for a given $M$ there only finitely many $N$ (up to homeomorphism) such that $\widehat{\pi_1 M}\iso \widehat{\pi_1 N}$.
\end{introthm}
We also use our analysis to distinguish graph manifolds from mixed or totally hyperbolic manifolds, which may be seen as an extension of \cite{WZ14} stating that the profinite completion `sees' which geometries are involved in the geometric decomposition of an aspherical 3-manifold.
\begin{introthm}{introGvM}
Let $M$ be a mixed or totally hyperbolic 3-manifold and let $N$ be a graph manifold. Then $\pi_1 M$ and $\pi_1 N$ do not have isomorphic profinite completions.
\end{introthm}
While our techniques give strong control over the Seifert fibre spaces involved in a 3-manifold, it is more difficult to detect the hyperbolic pieces in a satisfactory way; we can only detect the presence of a hyperbolic piece, without giving any information about it. When there are no Seifert-fibred pieces at all, we can extract some information about the configuration of the hyperbolic pieces.
\begin{introthm}{intropurehyp}
Let $M$, $N$ be aspherical manifolds with $\widehat{\pi_1 M}\iso \widehat{\pi_1 N}$ and with JSJ decompositions $(X, M_\bullet)$, $(Y, N_\bullet)$ all of whose pieces are hyperbolic. Then the graphs $X$ and $Y$ have equal numbers of vertices and edges and equal first Betti numbers. 
\end{introthm}
This paper is structured as follows. In Sections \ref{secPrelims}, \ref{secgrpsacting}, and \ref{secGofGs} we recall background material and the elements of profinite Bass-Serre theory. In Section \ref{secAcylindrical} we prove some results about certain actions on profinite trees which, while known to experts, have not appeared in detail in the literature. Section \ref{secJSJ} forms the technical core of the paper, proving many useful lemmas and ultimately the first part of Theorem \ref{introGM}. Theorem \ref{introGvM} is proved in Section \ref{secGvM}, and Theorem \ref{intropurehyp} is proved in Section \ref{secPureHyp}. In Section \ref{secproP} we describe how, in certain situations, we may obtain a pro-$p$ analogue of our JSJ decomposition theorem. Finally in Section \ref{secGMrigid} we complete the proof of Theorem \ref{introGM}.

\begin{cnv}
For notational convenience, we will adopt the following conventions.
\begin{itemize}
\item From Section \ref{secgrpsacting} onwards, profinite groups will generally be given Roman letters $A,G,H,...$; the handful of discrete groups which appear will usually either be groups with standard symbols or the fundamental group of a space, so we will not reserve a particular set of symbols for them.
\item Profinite graphs will be given capital Greek letters $\Gamma, \Delta, T,...$, and abstract graphs will be $X,Y,...$
\item A finite graph of (profinite) groups will be denoted ${\cal G}= (X,G_\bullet)$ where $X$ is a finite graph and $G_\bullet$ will be an edge or vertex group.
\item There is a divergence of notation between profinite group theorists, for whom $\Z[p]$ denotes the $p$-adic integers, and manifold theorists who use the same symbol for the cyclic group of order $p$. We will follow the former convention and use $\Z/p\Z$ or $\Z/p$ to denote the cyclic group.
\item For us, a graph manifold will be required to be non-geometric, i.e. not a single \SFS{} or a Sol-manifold, hence not a torus bundle. A totally hyperbolic manifold will be a 3-manifold all of whose JSJ pieces are hyperbolic, and a mixed manifold will be a manifold with both Seifert-fibred and hyperbolic JSJ pieces. All 3-manifolds will be orientable.
\item The symbols \nsgp[f], \nsgp[o], \nsgp[{\mathnormal p}] will denote `normal subgroup of finite index', `open normal subgroup', `normal subgroup of index a power of $p$' respectively; similar symbols will be used for not necessarily normal subgroups.
\item For two elements $g,h$ of a group, $g^h$ will denote $h^{-1}gh$. That is, conjugation will be a right action.
\end{itemize}
\end{cnv}

The author would like to thank Marc Lackenby for carefully reading this paper, and is grateful to Henry Wilton for discussions regarding Lemma \ref{malnormalcusps} and for drawing the author's attention to \cite{GDJZ15}. The author was supported by the EPSRC and by a Lamb and Flag Scholarship from St John's College, Oxford.
\section{Preliminaries}\label{secPrelims}
\subsection{Profinite completions}
We first briefly recall the notion of a profinite group and a profinite completion. For more detail, the reader is referred to \cite{RZ00}.
\begin{defn}
A {\em profinite group} is an inverse limit of a system of finite groups. If all these finite groups are from a particular class ${\cal C}$ of finite groups, the limit is said to be a {\em pro-$\cal C$ group}.

Given a group $G$, its {\em profinite completion} is the profinite group $\hat G$ defined as the inverse limit of all finite quotient groups of $G$:
\[\hat G = \varprojlim_{N\nsgp[f] G} G/N\]
The {\em pro-$\cal C$ completion} is the pro-$\cal C$ group $\hat G_{\cal C}$ defined as the inverse limit of those finite quotients of $G$ which are in the class $\cal C$.
\end{defn}

In this paper, we will only refer to those classes $\cal C$ comprising all finite groups whose orders are only divisible by primes from a certain set of primes $\pi$. The term `pro-$\cal C$' is then replaced by `pro-$\pi$'. When $\pi$ is the set of all primes, we of course get the full profinite completion. We further abbreviate `pro-$\{p\}$' to `pro-$p$'. The pro-$p$ completion is usually denoted \proP{G}.

Typically, we only take the profinite completion of a residually finite group; then the natural map $G\to \hat G$ will be injective. A {\em (topological) generating set} for a profinite group $H$ is a subset $X\subseteq H$ such that the subgroup abstractly generated by $X$ is dense in $H$.

The profinite completion was determined by the inverse system of all finite quotients. In fact, by a standard argument (see \cite{DFPR}) two finitely generated groups have the same profinite completion if and only if the sets of isomorphism types of finite groups which can arise as finite quotients of the two groups are the same. 

For a profinite group $G=\varprojlim Q_i$ where the $Q_i$ are finite, the kernels of the maps $G\to Q_i$ form a basis for the topology at 1, and so they are open subgroups of $G$. A deep result of Nikolov and Segal \cite{NS07} shows that for a finitely generated profinite group, all finite index subgroups are open; that is, the group structure determines a unique topology giving $G$ the structure of a profinite group. Thus any abstract isomorphism of finitely generated profinite groups is a topological isomorphism also.
   
The profinite completion of the integers plays, as one might expect, a central role in much of the theory. It is the inverse limit
\[\hat\Z=\varprojlim \Z/n\]
of all finite cyclic groups. The pro-$\pi$ completion, denoted \Z[\pi], is the inverse limit of those finite cyclic groups whose orders only involve primes from $\pi$. Because of the Chinese Remainder Theorem, these finite cyclic groups split as products of cyclic groups of prime power order, and these splittings are natural with respect to the quotient maps $\Z/mn\to\Z/n$. It follows that the profinite completion of $\Z$ splits as the direct product, over all primes $p$, of the rings of $p$-adic integers
\[\Z[p] = \varprojlim \Z/p^i\]
Similarly
\[\Z[\pi] = \prod_{p\in \pi} \Z[p]\]
This splitting as a direct product, which of course is not a feature of $\Z$ itself, means that ring-theoretically $\hat\Z$ behaves less well than $\Z$. In particular, there exist zero-divisors, which are precisely those elements which vanish under some projection to \Z[p]. The rings \Z[p] themselves do not have any zero-divisors, and $\Z$ injects into each \Z[p] (for every natural number $n$ and prime $p$, some power of $p$ does not divide $n$), so no element of \Z[] is a zero-divisor in $\hat \Z$. Furthermore, the group of units of $\hat \Z$ is rather large, being the inverse limit of the multiplicative groups of the rings $\Z/n$. In particular ${\rm Aut}(\hat\Z)$ is a profinite group much larger than ${\rm Aut}(\Z)$. 

Elements of profinite groups may be raised to powers with exponents in $\hat \Z$. For if $G$ is a profinite group and $x\in G$, the map $1\mapsto x$ extends to a map $\Z\to G$ by the universal property of \Z[], which by continuity extends to a map $\hat\Z\to G$ (that is, $\hat\Z$ is a free profinite group on the element 1). The image of $\lambda\in\Z$ under this map is then denoted $x^\lambda$. This operation has all the expected properties; see Section 4.1 of \cite{RZ00} for more details.

As we are discussing $\hat\Z$, this seems a fitting place to include the following easy lemma.
\begin{lem}\label{nonzerodiv}
If $\lambda\in\hat\Z$, $n\in\Z\smallsetminus\{0\}$, and $\lambda n\in \Z$, then $\lambda\in\Z$. In particular if $\lambda\in\hat\Z{}^{\!\times}$ then $\lambda=\pm 1$.
\end{lem}
\begin{proof}
Since $\lambda n$ is an integer congruent to 0 modulo $n$, there exists $l\in \Z$ such that $\lambda n=ln$. Since $n$ is not a zero-divisor in $\hat\Z$, it follows that $\lambda=l\in\Z$.
\end{proof}
\subsection{The profinite topology}
Whenever profinite properties of groups are discussed, it is usually necessary to have some control over subgroup separability. Here we recall prior results that will be used heavily, and often without comment, in the sequel.
\begin{defn}
The {\em (full) profinite topology} on a discrete group $G$ is the topology induced by the map $G\to\hat G$. A subset of $G$ is {\em separable} if and only if it is closed in the profinite topology.
\end{defn}
In particular, a group is residually finite if and only if the subset $\{1\}$ is separable. If every finite-index subgroup $H_1\sbgp[f] H$ of a subgroup $H\leq G$ is separable in $G$, then it follows that the natural map $\hat H\to \bar H^G$ is in fact an isomorphism. If this is the case we say that $G$ {\em induces the full profinite topology on $H$}. Seifert fibred and hyperbolic 3-manifold groups have very good separability properties; specifically they are LERF, meaning that every finitely generated subgroup is separable. In this case all finitely generated subgroups $H$ have $\hat H\iso \bar H$.
 
 Unfortunately graph manifold groups are not in general LERF; in fact in a sense a generic graph manifold group is not LERF. See \cite{BKS87}, \cite{niblowise}. However, those subgroups of primary concern are well behaved in the profinite topology. In particular:
\begin{thmquote}[Hamilton \cite{Ham01}]
Let $M$ be a Haken 3-manifold. Then the abelian subgroups of $\pi_1 M$ are separable in $\pi_1 M$ (and thus $\pi_1 M$ induces the full profinite topology on them).
\end{thmquote}
\begin{theorem}[Wilton and Zalesskii \cite{WZ10}]\label{JSJefficient}
Let $M$ be a closed, orientable, irreducible 3-manifold, and let $(X, M_\bullet)$ be the graph of spaces corresponding to the JSJ decomposition. Then the vertex and edge groups $\pi_1 M_\bullet$ are closed in the profinite topology on $\pi_1 M$, and $\pi_ 1 M$ induces the full profinite topology on them. 
\end{theorem}
\subsection{Seifert fibre spaces and graph manifolds}\label{SFSprelims}
In the literature there are two inequivalent definitions of what is meant by a `graph manifold'. In all cases, a graph manifold is an irreducible manifold whose JSJ decomposition consists only of Seifert-fibred pieces. The JSJ decomposition is by definition minimal, so the fibrings of adjacent pieces of the decomposition never extend across the union of those pieces. Some authors additionally require that a graph manifold is not geometric; i.e., it is not itself Seifert-fibred and is not a Sol manifold. Hence the JSJ decomposition is non-trivial and the graph manifold is not finitely covered by a torus bundle. In this paper we do impose this constraint. Furthermore we shall deal only with orientable graph manifolds.

At the other end of the spectrum, a manifold whose JSJ decomposition is non-trivial and has no Seifert-fibred pieces at all will be called `totally hyperbolic'. When the JSJ decomposition has at least one Seifert-fibred piece and at least one hyperbolic piece the manifold is called `mixed'.

We now briefly recall those facts about fundamental groups of Seifert fibre spaces which will be necessary. For a more full account, see \cite{scott83} or \cite{brinnotes}.

The fundamental group of a \SFS{} $M$ has a short exact sequence
\[ 1\to \gp{h}\to \pi_1 M\to \ofg[O] \to 1\]
where $O$ is the base orbifold and $h$ is the homotopy class of a regular fibre. The subgroup generated by this fibre (which may be finite in general) is normal in $\pi_1 M$. Moreover, either it is central in $\pi_1 M$ (when $O$ is orientable) or is contained in some index 2 subgroup of $\pi_1 M$ in which it is central (when $O$ is non-orientable). The orbifold fundamental group \ofg[O] is a Fuchsian group. The \SFS{}s arising in the JSJ decomposition of a graph manifold have boundary consisting of at least one incompressible torus comprised of fibres; in this case the subgroup generated by $h$ is infinite cyclic and \ofg[O] is a free product of cyclic groups. 

The base orbifold has either positive, zero, or negative Euler characteristic. If the Euler characteristic is negative, then there is a unique maximal normal cyclic subgroup of $\pi_1 M$, called the {\em (canonical) fibre subgroup}. 

No positive characteristic base orbifolds can occur when the boundary is incompressible and non-empty. The \SFS{}s with boundary which have Euclidean base orbifold are precisely $\sph{1}\times\sph{1}\times\I$ and the orientable $\I$-bundle over a Klein bottle. The first of these can never arise in a JSJ decomposition of anything other than a torus bundle and we will henceforth ignore it. The second has two maximal normal cyclic subgroups (`fibre' subgroups), one central (with quotient the infinite dihedral group) and one not central (with quotient \Z, here being the fundamental group of a M\"obius band). We call the first of these fibre subgroups the {\em canonical} fibre subgroup. We refer to such pieces of the JSJ decomposition of a graph manifold as `minor', and to those with base orbifold of negative characteristic as `major' (the terms `small \SFS' and `large \SFS' being already entrenched in the literature with a rather different meaning). 

By definition, all pieces of the JSJ decomposition of a graph manifold are \SFS{}s, and by minimality of the decomposition the fibres of two adjacent Seifert fibred pieces do not match, even up to isotopy. Thus the canonical fibre subgroups of adjacent \SFS{}s intersect trivially in the fundamental group of the graph manifold; if one piece is minor, neither of its two fibre subgroups intersect the fibre subgroup of the adjacent piece non-trivially.

Note that two minor pieces can never be adjacent; for each of these having only one boundary component, the whole graph manifold would then be just two minor pieces glued together. Each has an index 2 cover which is a copy of $\sph{1}\times\sph{1}\times\I$, so our graph manifold would be finitely covered by a torus bundle, and would thus be either a Euclidean, Nil or Sol manifold; but we required our graph manifolds to be non-geometric.

Many of these properties still hold in the profinite completion; for instance, when the base orbifold is of negative Euler characteristic, Theorem 6.4 of \cite{Wilkes15} guarantees that we still have a unique maximal procyclic subgroup which is either central or is central in an index 2 subgroup. We may directly check that the profinite completion of the Klein bottle group also still has just two maximal normal procyclic subgroups, one of which is central. Hence our notion of (canonical) fibre subgroup, as a maximal procyclic group with the above property, carries over to the profinite world.

Seifert fibre spaces are well-controlled by their profinite completions. 
\begin{theorem}[Wilkes \cite{Wilkes15}]
Let $M$ be a closed orientable Seifert fibre space. Then:
\begin{itemize}
\item If $M$ has the geometry $\sph{3}$, $\sph{2}\times\R$, $\E^3$, {\rm Nil}, or $\SLR$ then $M$ is profinitely rigid.
\item If $M$ has the geometry \HR, and is therefore a surface bundle over a hyperbolic surface $\Sigma$ with periodic monodromy $\phi$, those orientable 3-manifolds with the same profinite completion as $M$ are precisely the surface bundles over $\Sigma$ with monodromy $\phi^k$, for $k$ coprime to the order of $\phi$.
\end{itemize}
\end{theorem}
The non-rigid examples of geometry \HR{} were found by Hempel \cite{hempel14}.

For Seifert fibre spaces with boundary, such as those arising in the JSJ decomposition of a graph manifold, we naturally require some conditions on the boundary. 
\begin{defn}\label{exauto}
Let $O$ be an orientable 2-orbifold with boundary, with fundamental group
\[B= \left< a_1,\ldots, a_r, e_1, \ldots, e_s, u_1, v_1, \ldots, u_g,v_g\mid a_i^{p_i}\right>\]
where the boundary components of $O$ are represented by the conjugacy classes of the elements $e_1, \ldots, e_s$ together with
\[ e_0 = \left(a_1 \cdots a_r e_1\cdots e_s [u_1, v_1]\cdots [u_g,v_g]\right)^{-1} \]
Then an {\em exotic automorphism of $O$ of type $\mu$} is an automorphism $\psi\colon\hat B\to\hat B$ such that $\psi(a_i) \sim a_i^\mu$ and $\psi(e_i)\sim e_i^\mu$ for all $i$, where $\sim$ denotes conjugacy in $\hat B$.
Similarly, let $O'$ be a non-orientable 2-orbifold with boundary, with fundamental group 
\[B'= \left< a_1,\ldots, a_r, e_1, \ldots, e_s, u_1,\ldots, u_g\mid a_i^{p_i}\right>\]
where the boundary components of $O'$ are represented by the conjugacy classes of the elements $e_1, \ldots, e_s$ together with
\[ e_0 = \left( a_1 \cdots a_r e_1\cdots e_su_1^2\cdots u_g^2\right)^{-1} \]
Let $o\colon \hat B'\to \{\pm 1\}$ be the orientation homomorphism of $O'$. Let $\sigma_0, \ldots, \sigma_s\in\{\pm1\}$. Then an {\em exotic automorphism of $O'$ of type $\mu$ with signs $\sigma_0, \ldots, \sigma_s$} is an automorphism $\psi\colon\hat B'\to\hat B'$ such that $\psi(a_i) \sim a_i^\mu$ and $\psi(e_i)= (e_i^{\sigma_i\mu})^{g_i}$ for all $i$ where $o(g_i)=\sigma_i$ for all $i$.
\end{defn}
\begin{rmk}
The reader may find the term `exotic automorphism of $O$' a little jarring as the automorphism really acts on $\hat B$. This name was chosen to emphasize that this notion depends on an identification of the group as the fundamental group of a specific orbifold, with specific elements representing boundary components. For example, an exotic automorphism of a three-times punctured sphere is a rather different thing from an exotic automorphism of a once-punctured torus: even though both groups are free of rank 2, there are different numbers of boundary components to consider. 

For the same reason, there is a canonical map to a cyclic group of order 2 giving the orientation homomorphism- this is not a characteristic quotient of a free group, but is uniquely defined when an identification with an orbifold group is chosen.
\end{rmk}
\begin{theorem}[Wilkes \cite{Wilkes15}]\label{Wilkeswbdy}
Let $M, N$ be Seifert fibre spaces with boundary components $\bdy M_1,\ldots,\bdy M_n,\bdy N_1,\ldots,\bdy N_n$. Suppose $\Phi$ is an isomorphism of group systems
\[ \Phi: (\widehat{\pi_1 M};\widehat{\pi_1\bdy M_1},\ldots, \widehat{\pi_1\bdy M_n})\to (\widehat{\pi_1 N};\widehat{\pi_1\bdy N_1},\ldots, \widehat{\pi_1\bdy N_n}) \]
Then:
\begin{itemize}
\item If $M$ is a minor \SFS, then $M\iso N$.
\item If $M$ is a major \SFS, then the base orbifolds of $M$, $N$ may be identified with the same orbifold $O$ such that $\Phi$ splits as an isomorphism of short exact sequences
\[\begin{tikzcd}
1\ar{r} & \hat\Z\ar{r} \ar{d}{\cdot \lambda} & \widehat{\pi_1 M} \ar{r} \ar{d}{\Phi} & \widehat{\ofg[O]} \ar{r} \ar{d}{\phi} & 1 \\
1\ar{r} & \hat\Z\ar{r} & \widehat{\pi_1 N} \ar{r} & \widehat{\ofg[O]} \ar{r} & 1
\end{tikzcd}\]
where $\lambda$ is some invertible element of $\hat\Z$ and $\phi$ is an exotic  automorphism of $O$ of type $\mu$.

Hence if $N$ is a surface bundle over the circle with fibre a hyperbolic surface $\Sigma$ with periodic monodromy $\psi$, then $M$ is also such a surface bundle with monodromy $\psi^k$ where $k$ is congruent to $\kappa=\mu^{-1}\lambda$ modulo the order of $\psi$.   
\end{itemize} 
\end{theorem}
\begin{rmk}
Since this precise theorem statement does not appear in \cite{Wilkes15} in quite this form, we should comment on how it arises, and in particular the point that $\kappa$ is $\mu^{-1}\lambda$ rather than just $\lambda\mu$.

First we remark that this theorem follows immediately from the case of orientable base orbifold. Note that $\Phi$ preserves the canonical index 2 subgroup given by the centraliser of the fibre, so the corresponding 2-fold covers $\tilde M$ and $\tilde N$ are related in the above fashion, for some labelling of the boundary components of each. The result then follows, noting that an automorphism of the base which induces an exotic automorphism on the index 2 orientation subgroup is itself an exotic automorphism.

So consider the case of orientable base orbifold. Let $B=\ofg[O]$. The fact that an isomorphism of fundamental groups gives an isomorphism of short exact sequences is the content of Theorem 6.4 of \cite{Wilkes15}. The profinite groups in question are then central extensions, given by cohomology classes 
\[\zeta_M, \zeta_N\in H^2(\hat B, \hat\Z) \]
This latter cohomology group is an abelian group of exponent dividing the order of the monodromy. The statement concerning surface bundles is equivalent to the statement that \[\zeta_N=\mu^{-1}\lambda \zeta_M\] 
To see this equivalence, one could write out presentations of the groups and hence compute the cohomology classes directly from the characterisation in Section 6.2 of \cite{Wilkes15}. See also Section 5 of \cite{hempel14}, where the computation of the fibre invariants of a surface bundle with monodromy $\psi^k$ is carried out. Note that raising the monodromy to a power $k$ actually raises the fibre invariants $(p_i, q_i)$, and hence the cohomology class, of the space to a power {\em inverse to} $k$ modulo the order of $\psi$- hence the order of $N$ and $M$ in the last part of the theorem.

From the proof of Theorem 6.7 of \cite{Wilkes15}, the automorphism $\phi$ of $\hat B$ is an exotic automorphism of type $\mu$ for some $\mu$, and the action of such a $\phi$ on cohomology is multiplication by $\mu$. Let $\Gamma$ be a central extension of $\hat B$ by $\hat\Z$ corresponding to the cohomology class $\mu\zeta_N$. Since the action of $\phi^{-1}$ on cohomology is multiplication by $\mu^{-1}$, and noting that this action is {\em contravariant}, we have  
\[(\phi^{-1})^\ast (\mu\zeta_N)= \zeta_N \]
and hence we have a short exact sequence of isomorphisms
\[\begin{tikzcd}
1\ar{r} & \hat\Z\ar{r} \ar{d}{\id} & \widehat{\pi_1 N} \ar{r} \ar{d}{\iso} & \hat B \ar{r} \ar{d}{\phi^{-1}} & 1 \\
1\ar{r} & \hat\Z\ar{r} & \Gamma \ar{r} & \hat B \ar{r} & 1
\end{tikzcd}\]
Precomposing this with the short exact sequence from the theorem there is an isomorphism
\[\begin{tikzcd}
1\ar{r} & \hat\Z\ar{r} \ar{d}{\cdot \lambda} & \widehat{\pi_1 M} \ar{r} \ar{d}{\iso} & \hat B \ar{r} \ar{d}{\id} & 1 \\
1\ar{r} & \hat\Z\ar{r} & \Gamma \ar{r} & \hat B \ar{r} & 1
\end{tikzcd}\]
which, since the map on $\hat\Z$ gives a {\em covariant} map on cohomology, says that the cohomology class representing $\Gamma$ is in fact $\lambda\zeta_M$. Hence
\[ \lambda\zeta_M=\mu\zeta_N\] 
as was claimed. 
\end{rmk}
\begin{defn}\label{defHempelpair}
If $M,N$ are as in the latter case of the above theorem, we say that $(M,N)$ is a {\em Hempel pair of scale factor $\kappa$}, where $\kappa= \mu^{-1}\lambda$. Note that $\kappa$ is only well-defined modulo the order of $\phi$, which may be taken to be the lowest common multiple of the orders of the cone points of $M$. Note that a Hempel pair of scale factor $\pm 1$ is a pair of homeomorphic \SFS{}s.
\end{defn}
\section{Groups acting on profinite graphs}\label{secgrpsacting}
\subsection{Profinite graphs}
We state here the definitions and basic properties of profinite graphs for convenience. The sources for this section are \cite{RZup}, Section 1 or \cite{RZ00p}, Sections 1 and 2. See these references for proofs and more detail.
\begin{defn}
An {\em abstract graph} $X$ is a set with a distinguished subset $V(X)$ and two retractions $d_0,d_1:X\to V(X)$. Elements of $V(X)$ are called vertices, and elements of $E(X)=X\smallsetminus V(X)$ are called edges. Note that a graph comes with an orientation on each edge.

If an abstract graph is in addition a profinite space $\Gamma$ (that is, an inverse limit of finite discrete topological spaces), $V(\Gamma)$ is closed and $d_0,d_1$ are in addition continuous, then $\Gamma$ is called a {\em profinite graph}. Note that $E(\Gamma)$ may not be closed.
\end{defn}
 A morphism of graphs $X$, $Y$ is a function $f:X\to Y$ such that $d_i f = fd_i$ for each $i$. In particular, $f$ sends vertices to vertices, but may not send edges only to edges. A morphism of profinite graphs $\Gamma$, $\Delta$ is a map $f:\Gamma\to\Delta$ which is a morphism of abstract graphs and is continuous.
 
   A profinite graph may equivalently be described as an inverse limit of finite abstract graphs $X_i$, and $V(\Gamma)=\varprojlim V(X_i)$.  If $E(\Gamma)$ happens to be closed, we can choose the inverse system $X_i$ so that the transition maps send edges to edges, and then $E(\Gamma)=\varprojlim E(X_i)$. If $Y$ is another finite graph, any morphism of $\Gamma$ onto $Y$ factors through some $X_i$.
 
A {\em path of length $n$} in a graph $X$ is a morphism into $X$ of a finite graph consisting of $n$ edges $e_1,\ldots, e_n$ and $n+1$ vertices $v_0,\ldots, v_n$ (with some choice of orientations on the edges) such that the endpoints of each edge $e_i$ are $v_{i-1}, v_i$. Note that such a morphism into a profinite graph is automatically continuous. 

An abstract graph $X$ is {\em path-connected} if any two vertices lie in the image of some path in $X$. A profinite graph $\Gamma$ is {\em connected} if any finite quotient graph is path-connected. Equivalently a connected profinite graph is the inverse limit of path-connected finite graphs.

A path-connected profinite graph is connected. If the set of edges of a connected profinite graph is closed, then each vertex must have an edge incident to it.

\begin{example}[A connected profinite graph with a vertex with no edge incident to it]
Let $I=\N\cup\tilde\N\cup\{\infty\}$ be the one-point compactification of two copies $\N,\tilde\N$ of the natural numbers. Let $V(I)=\N\cup\{\infty\}$, and define $d_0,d_1$ by $d_i(n)=n$ for $n\in\N$, $d_i(\infty)=\infty$, and $d_0(\tilde n)=n$, $d_1(\tilde n)=n+1$ for $\tilde n\in\tilde\N$. This is the inverse limit of the system of finite line segments $I_n$ of length $n$, where the maps $I_{n+1}\to I_n$ collapse the final edge to the vertex $n$. Hence $I$ is connected, but no edge is incident to $\infty$. 
\end{example}
\begin{example}[A connected graph whose proper connected subgraphs are finite]
Consider the Cayley graphs $C_n={\rm Cay}(\Z/n, 1)$ with the natural maps $C_{mn}\to C_n$   coming from the natural group homomorphisms. Then $\Gamma=\varprojlim C_n$ is the Cayley graph of $\hat \Z$ with respect to the generating set $1$ (see below). If $\Delta$ is a connected proper subgraph of $\Gamma$, then the image $\Delta_n$ of $\Delta$ in some $C_n$ is not all of $C_n$; hence it is a line segment in $C_n$. In each $C_{mn}$, the preimage of this line segment is $m$ disjoint copies of $\Delta_n$; by connectedness $\Delta_{mn}$ is precisely one of these. That is, $\Delta_{mn}\to\Delta_n$ is an isomorphism for all $m$. Thus $\Delta =\varprojlim \Delta_{mn}$ is just a line segment of finite length. 
\end{example}
\begin{example}[Cayley graph]
Let $G$ be a profinite group and $X$ be a closed subset of $G$, possibly containing 1. The {\em Cayley graph} of $G$ with respect to $X$ is the profinite graph $\Gamma=G\times(X\cup\{1\})$, where $V(\Gamma)=G\times\{1\}$, $d_0(g,x)=(g,1)$ and $d_1(g,x)=(gx,1)$. 

The Cayley graph $\Gamma$ is the inverse limit of the Cayley graphs of finite quotients of $\Gamma$, with respect to the images of the set $X$ in these quotients. The set $X$ generates $G$ topologically if and only if its image in each finite quotient is a generating set; hence $\Gamma$ is connected if and only if every finite Cayley graph is connected, if and only if $X$ generates $G$.
\end{example}
\subsection{Profinite trees}
As one might expect from the reduced importance of paths in the theory of profinite graphs, a `no cycles' condition does not give a good definition of `tree'. Instead a homological definition is used. Let $\mathbb{F}_p$ denote the finite field with $p$ elements for $p$ a prime. The following definitions and statements can be found in Section 2 of \cite{RZ00p}. For more on profinite modules and chain complexes, see Chapters 5 and 6 of \cite{RZ00}.
\begin{defn}
 Given a profinite space $X=\varprojlim X_i$ where the $X_i$ are finite spaces, define the {\em free profinite $\mathbb{F}_p$-module on $X$} to be
 \[ [[\mathbb{F}_p X]] =\varprojlim [\mathbb{F}_p X_i] \]
the inverse limit of the free $\mathbb{F}_p$-modules with basis $X_i$. Similarly for a pointed profinite space $(X,\ast)=\varprojlim (X_i,\ast)$ define
\[ [[\mathbb{F}_p (X,\ast)]] =\varprojlim [\mathbb{F}_p (X_i,\ast)] \]
These modules satisfy the expected universal property, that a map from $X$ to a profinite $\mathbb{F}_p$-module $M$ (respectively, a map from $(X,\ast)$ to $M$ sending $\ast$ to 0) extends uniquely to a continuous morphism of modules from the free module to $M$.
\end{defn} 
\begin{defn}
Let $\Gamma$ be a profinite graph. Let $(E^\ast(\Gamma),\ast)$ be the pointed profinite space $\Gamma/V(\Gamma)$ with distinguished point the image of $V(\Gamma)$. Consider the chain complex
\[ 0\to [[\mathbb{F}_p(E^\ast(\Gamma),\ast)]]\stackrel{\delta}{\to}[[\mathbb{F}_p V(\Gamma)]]\stackrel{\epsilon}{\to}\mathbb{F}_p\to 0\]
where the map $\epsilon$ is the evaluation map and $\delta$ sends the image of an edge $e$ in $(E^\ast(\Gamma),\ast)$ to $d_1(e)-d_0(e)$. Then define $H_1(\Gamma,\mathbb{F}_p)={\rm ker}(\delta)$ and $H_0(\Gamma, \mathbb{F}_p) = {\rm ker}(\epsilon)/{\rm im}(\delta)$.
\end{defn}
\begin{prop}\label{homologyfunct}
Let $\Gamma$ be a profinite graph. 
\begin{itemize}
\item $\Gamma$ is connected if and only if $H_0(\Gamma, \mathbb{F}_p)=0$
\item The homology groups $H_0, H_1$ are functorial, and if $\Gamma=\varprojlim \Gamma_i$, then 
\[H_j(\Gamma,\mathbb{F}_p)= \varprojlim H_j(\Gamma_i,\mathbb{F}_p) \]
\end{itemize}
\end{prop}
\begin{defn}
A profinite graph $\Gamma$ is a {\em pro-$p$ tree}, or simply $p$-tree, if $\Gamma$ is connected and $H_1(\Gamma, \mathbb{F}_p)=0$. If $\pi$ is a non-empty set of primes then $\Gamma$ is a {\em $\pi$-tree} if it is a $p$-tree for every $p\in\pi$. We say simply `profinite tree' if $\pi$ consists of all primes.
\end{defn}
Note that a finite graph is a $p$-tree if and only if it is an abstract tree. It follows immediately from Proposition \ref{homologyfunct} that an inverse limit of finite trees is a $p$-tree. However this is not the only source of $p$-trees, and the notion of $p$-tree is not independent of the prime $p$.
\begin{example}
Let $\Gamma$ be the Cayley graph of $\hat\Z$ with respect to the generator 1, written as the inverse limit of cycles $C_n$ of length $n$. For a finite set $X$, 
\[ [[\mathbb{F}_p (X\cup\{\ast\},\ast)]]=[[\mathbb{F}_p X]]=[\mathbb{F}_p X] \]
so for each $n$, $H_1(C_n,\mathbb{F}_p)=\mathbb{F}_p$ as this is now the simplicial homology of the realisation of $C_n$ as a topological space. Thus the map $H_1(C_{mn},\mathbb{F}_p)\to H_1(C_n,\mathbb{F}_p)$ is just multiplication by $m$. In particular, $C_{pn}\to C_n$ induces the zero map on homology. It follows that the inverse limit of these groups is trivial, i.\ e.
\[ H_1(\Gamma,\mathbb{F}_p) = \varprojlim H_1(C_n,\mathbb{F}_p) =0\]
so that $\Gamma$ is a $p$-tree.

Now let $\Gamma_q$ be the Cayley graph of the $q$-adic integers \Z[q] with respect to 1, where $q$ is another prime. Again $\Gamma_q$ is the inverse limit of the cycles $C_{q^i}$ of length $q^i$, and the maps induced on homology by $C_{q^j}\to C_{q^i}$ for $j> i$ are multiplication by $q^{j-i}$. If $q=p$ these are the zero map, so again  
\[ H_1(\Gamma_q,\mathbb{F}_p) = \varprojlim H_1(C_{q^i},\mathbb{F}_p) =0\]
so that $\Gamma_q$ is a $q$-tree. On the other hand, if $q\neq p$, multiplication by $q^{j-i}$ is an isomorphism from $\mathbb{F}_p$ to itself, so that
\[ H_1(\Gamma_q,\mathbb{F}_p) = \varprojlim H_1(C_{q^i},\mathbb{F}_p) \iso\mathbb{F}_p\]
and $\Gamma_q$ is {\em not} a $p$-tree for $p\neq q$.  
\end{example}

It transpires (Theorem 1.4.3 of \cite{RZup}) that the Cayley graph of any free pro-$\pi$ group with respect to a free basis (that is, a pro-$\pi$ group satisfying the appropriate universal property in the category of pro-$\pi$ groups) is a $\pi$-tree. Note that the examples above show that a profinite group acting freely on a profinite tree need not be free profinite. For the subgroup $\Z[p]\leq\hat\Z$ acts freely on the above Cayley graph, but is not a free object in the category of all profinite groups (for instance, if $q$ is another prime, the map $1\mapsto 1\in \Z/q$ does not extend to a map $\Z[p]\to\Z/q$). It turns out (Theorem 3.1.2 of \cite{RZup}) that such a group will instead be {\em projective} in the sense of category theory. A pro-$p$ group acting freely on a pro-$p$ tree is a free pro-$p$ group however (Theorem 3.4 of \cite{RZ00p}), as in the category of pro-$p$ groups, any projective group is free pro-$p$. This is one of the ways in which the pro-$p$ theory is more amenable than the general theory.

Some of the topological properties of abstract trees do carry over well to the world of profinite trees:
\begin{prop}[Propositions 1.5.2, 1.5.6 of \cite{RZup}]
Let $T$ be a $\pi$-tree.
\begin{itemize}
\item Every connected profinite subgraph of $T$ is a $\pi$-tree
\item Any intersection of $\pi$-subtrees of $T$ is a (possibly empty) $\pi$-subtree
\end{itemize}
\end{prop}
It follows that for every subset $W$ of a $\pi$-tree $T$ there is a unique smallest subtree of $T$ containing $W$. If $W$ consists of two vertices $v,w$ then this smallest subtree is called the {\em geodesic} from $v$ to $w$ and is denoted $[v,w]$. Note that if a profinite tree $T$ is path-connected, hence is also an abstract tree, then $[v,w]$ will coincide with the usual notion of geodesic, a shortest path from $v$ to $w$. However there is no requirement for a geodesic to be a path. For instance, our above analysis of the connected subgraphs of the profinite tree ${\rm Cay}(\hat\Z,1)$ shows that if $\lambda,\mu\in\hat\Z$ then either $\lambda-\mu\in\Z$ or the geodesic $[\lambda,\mu]$ is the entire tree ${\rm Cay}(\hat\Z,1)$.

\subsection{Group actions on profinite trees}
The theory of profinite groups acting on profinite trees is less tractable than the classical theory, but still parallels it in many respects. In this section we will recall results from the unpublished book \cite{RZup}, and prove others which will be of use. The theory was originally developed in \cite{GR78}, \cite{Zal89}, \cite{ZM89a}, and \cite{ZM89b}, and the pro-$p$ version of the theory may be found in published form in \cite{RZ00p}.
\begin{defn}
A profinite group $G$ is said to {\em act} on a profinite graph $\Gamma$ if $G$ acts continuously on the profinite space $\Gamma$ in such a way that 
\[g\cdot d_i(x) = d_i(g\cdot x)\]
for all $g\in G, x\in \Gamma, i=0,1$. For each $x\in\Gamma$, the stabiliser $\{g\in G\,|\, g\cdot x=x\}$ will be denoted $G_x$. For subsets $X\subseteq \Gamma$, $H\subseteq G$, the set of points in $X$ fixed by every element of $H$ will be denoted $X^H$.
\end{defn}
Note that given an edge $e$, we cannot have $g\cdot e =e$ without fixing both endpoints of $e$, as $g\cdot d_i(e) = d_i(e)$ for each $i$. In particular, the qualification `without inversion' applied to group actions in the classical theory of \cite{Serre} is here subsumed in the definition. 

If $G$ acts on $\Gamma$, the quotient space $G\backslash\Gamma$ is a well-defined profinite graph. As one might expect, such an action may be represented as an `inverse limit of finite group actions on finite graphs'. More precisely, 
\begin{defn}
Let a profinite group $G$  act on a profinite graph $\Gamma$. A decomposition of $\Gamma$ as an inverse limit of finite graphs $\Gamma = \varprojlim X_i$ is said to be a {\em $G$-decomposition} if $G$ acts on each $X_i$ in such a way that all quotient maps $\Gamma\to X_i$ and transition maps $X_j\to X_i$ are $G$-equivariant.
\end{defn}

Indeed, the following is true
\begin{lem}[Lemma 1.2.1 of \cite{RZup}]\label{Gdecomp}
Let a profinite group $G$ act on a profinite graph $\Gamma$. Then the graphs $\{N\backslash \Gamma\,|\, N\nsgp[o] G\}$ form an inverse system and
\[\Gamma = \varprojlim_{N\nsgp[o] G} N\backslash \Gamma\]
\end{lem}
\begin{proof}
The first statement is clear. By the universal property of inverse limits, we have a natural continuous surjection
\[\Gamma\twoheadrightarrow \varprojlim_{N\nsgp[o] G} N\backslash \Gamma\]
and it remains to show that this is injective. If $v,w\in \Gamma$ are identified in every $N\backslash \Gamma$, then for all $N$ there is some $n\in N$ such that $n\cdot v =w$. Thus the closed subsets 
\[\{n\in N \,|\, n\cdot v = w\}\]
of $G$ are all non-empty, and the collection $\{N\}$ is closed under finite intersections; so by compactness of $G$, their intersection is non-empty, so there is some $g\in\bigcap N$ such that $g\cdot v=w$; but the intersection of all $N$ is trivial, so $v=w$.
\end{proof}

\begin{defn}
If $G$ is a profinite group acting on a profinite tree $T$, the action is said to be:
\begin{itemize}
\item {\em faithful}, if the only element of $G$ fixing every vertex of $T$ is the identity;
\item {\em irreducible}, if no proper subtree of $T$ is invariant under the action of $G$;
\item {\em free} if $G_x=1$ for all $x\in \Gamma$.
\end{itemize}
\end{defn}
Note that a group $G$ acts freely on its Cayley graph with respect to any closed subset $X$, with quotient the `bouquet of circles' on the pointed profinite space $(X\cup\{1\},1)$, i.\ e.\ the profinite graph $\Gamma(X\cup\{1\},1)$ with vertex space $\{1\}$.

Faithful and irreducible actions are the most important actions; indeed, given a group $G$ acting on a profinite tree $\Gamma$, we can quotient $G$ by the kernel $\bigcap G_x$ of the action (i.\ e. those group elements fixing all of $\Gamma$) and then pass to a minimal $G$-invariant subtree to get a faithful irreducible action. Such a subtree exists by:
\begin{prop}[Proposition 1.5.9 of \cite{RZup}; Lemma 3.11 of \cite{RZ00p}]
Let a profinite group $G$ act on a pro-$\pi$ tree $\Gamma$. Then there exists a minimal $G$-invariant subtree $\Delta$ of $\Gamma$, and if $|\Delta|>1$, it is unique.   
\end{prop}

\begin{theorem}[Theorem 3.1.5 of \cite{RZup}]\label{fixedptsubtree}
Suppose a pro-$\pi$ group $G$ acts on a $\pi$-tree $T$. Then the set $T^G$ of fixed points under the action of $G$ is either empty or a $\pi$-subtree of $G$.
\end{theorem}

\begin{theorem}[Theorem 3.1.7 of \cite{RZup}]\label{finitegroups}
Any finite group acting on a profinite tree fixes some vertex.
\end{theorem}

\begin{prop}\label{powersfix}
Let $G$ act on a profinite tree $T$, and let $x\in G$. If $x^\lambda$ fixes a vertex $v$ for some $\lambda\in \hat \Z\smallsetminus\{0\}$ which is not a zero-divisor, then $x$ also fixes some vertex.
\begin{proof}
Let $S$ be the subtree of $T$ fixed by $x^\lambda$; it is non-empty by assumption. Consider the action of $C=\overline{<\!x\!>}$ on $S$. The closed (normal) subgroup of $C$ generated by $x^\lambda$ acts trivially on $S$, so there is a quotient action of $C/\overline{<\!x^\lambda\!>}$ on $S$. Now this quotient group is $\hat\Z/\lambda\hat\Z$. Using the splitting $\hat\Z=\prod \Z[p]$, we find
\[\hat\Z/\lambda\hat\Z = \prod \Z[p]/\pi_p(\lambda)\Z[p]\]
where $\pi_p$ is the projection onto each factor. No $\pi_p(\lambda)$ is zero because $\lambda$ is not a zero-divisor in $\hat\Z$. Closed subgroups of \Z[p] are finite index or trivial (see Proposition 2.7.1 of \cite{RZ00}), so $\hat\Z/\lambda\hat\Z$ is a direct product of finite cyclic groups.

The subtree fixed by a direct product of groups is the intersection of the subtrees fixed by each group. Any finite product of finite cyclic groups fixes some vertex by Theorem \ref{finitegroups}, so the subtrees fixed by each finite cyclic group are a collection of non-empty closed subsets of $S$ with the finite intersection property. By compactness, the intersection of all of them is non-empty; but this is the subtree fixed by $C/\overline{<\!x^\lambda\!>}$, which is the same as the subtree fixed by $x$. So $x$ fixes some vertex of $S$ (hence of $T$).
\end{proof}
\end{prop}

\begin{rmk}
The condition that $\lambda$ is not a zero-divisor is necessary. For instance, $\hat \Z$ (written multiplicatively with generator $x$) acts freely on its Cayley graph, but if $\lambda\mu=0$ for $\lambda, \mu$ non-zero, then $x^\mu$ fixes no vertex but $(x^\mu)^\lambda$ is the identity. Recall however that no element of $\Z$ is a zero divisor in $\hat\Z$, so that the Proposition applies in particular when $\lambda\in\Z$.
\end{rmk}

\begin{prop}[Proposition 3.2.3(b) of \cite{RZup}]\label{abelianactingontree}
Let $G$ be an abelian profinite group acting faithfully and irreducibly on a profinite tree. Then $G$ acts freely and $G\iso \Z[\pi]$ for some set of primes $\pi$.
\end{prop}

\section{Acylindrical actions}\label{secAcylindrical}
Actions on profinite trees are particularly malleable when the action is acylindrical.
\begin{defn}
Let a profinite group $G$ act on a profinite tree $T$. The action is {\em $k$-acylindrical} if the stabiliser of any injective path of length greater than $k$ is trivial.
\end{defn}
For instance an action with trivial edge stabilisers is 0-acylindrical. In \cite{WZ14} Wilton and Zalesskii exploited the fact that if edge groups are malnormal in the adjacent vertex groups then the action on the standard graph is 1-acylindrical. 

We now prove some results about acylindrical actions on profinite trees. The following lemma is taken from the Appendix to \cite{HeZ12} and will be used to remove some of the pathologies associated with profinite graphs; we reproduce it here for completeness.
\begin{lem}\label{boundeddiameter}
Let $\Gamma$ be a profinite graph in which there are no paths longer than $m$ edges for some integer $m$. Then the connected components of $\Gamma$ (as a profinite graph) are precisely the path components (that is, the connected components as an abstract graph). In particular if $\Gamma$ is connected then it is path-connected.
\begin{proof}
First define the composition of two binary relations $R,S$ on a set $X$ to be the relation that $xRSy$ if and only if there exists $z$ such that $xRz$ and $zSy$; inductively define $R^{n+1} = R^n R$. Further define $R^{\rm op}$ to be the relation that $xR^{\rm op}y$ if and only if $yRx$. Let $\Delta=\{(x,x)\in X\times X\}$ be the identity relation. 

For an abstract graph $\Gamma$ define $R_0 = \{(x,y)\in \Gamma\times \Gamma\,|\, d_1(x)=d_0(y)\}$ (where $x,y$ could be vertices), and set $R=R_0\cup\Delta\cup R_0^{\rm op}$. If $xR^n y$ then there is some path of length at most $n$ containing $x$ and $y$; if there is a path of $n$ edges containing $x$ and $y$ then $xR^{2n+1}y$. This discrepancy is due to our convention that graphs are oriented, so we may need to include vertices in addition to the edges in a path of length $n$ to get a chain of $R$-related elements of $\Gamma$. The path-components of $\Gamma$ are then the equivalence classes of the equivalence relation $S=\bigcup_n R^n$. 

Now in our profinite graph $\Gamma$, we have $S=R^n$ for some $n$ as there is a uniform bound on the length of paths in $\Gamma$. One can show that the continuity of the maps $d_0,d_1$ and compactness of $\Gamma$ imply that $R$, and all $R^n$, are closed compact subsets of $\Gamma\times\Gamma$. In particular, the equivalence classes of $S=R^n$ are closed subsets of $\Gamma$; that is, the path-components of $\Gamma$ are closed. The quotient profinite graph $\Gamma/S$ has no edges, hence its maximal connected subgraphs are points. Thus connected components of $\Gamma$ (as a profinite graph) are contained in, hence equal to, a path-component of $\Gamma$.
\end{proof}
\end{lem}
\begin{prop}\label{longgeods}
Let a profinite group $G$ act irreducibly on a profinite tree $T$. Then either $T$ is a single vertex or it contains paths of arbitrary length.
\begin{proof}
Assume that there is a bound on the lengths of paths in $T$; then by the previous result $T$ is path-connected, hence is a tree when considered as an abstract graph. Consider two geodesics of maximal length in this tree. These must intersect in at least their midpoint (a vertex if the length is even or an edge if odd); for otherwise one can easily construct a longer path. Hence the intersection of all such maximal geodesics is non-empty. This intersection is a profinite subtree of $T$ invariant under the action of $G$, which must permute these maximal geodesics. Hence if the action of $G$ is irreducible, $T$ is equal to the intersection of its maximal geodesics; hence $T$ is merely a path of length $n$. But then the midpoint of this path is fixed under the action of $G$; hence this middle vertex (or edge) is the whole of $T$. If $G$ fixes an edge, then by our conventions that graphs are oriented, then it fixes each endpoint. Hence $T$ is a single vertex.
\end{proof}
\end{prop}

\begin{clly}\label{samepathcpt}
Let a profinite group $G$ act $k$-acylindrically on a profinite tree $T$ for some $k$. Suppose $g\in G\smallsetminus\{1\}$ fixes two vertices $v,w$. Then $v,w$ are in the same path-component of $T$.
\begin{proof}
By Theorem \ref{fixedptsubtree}, the set $T^g$ of points fixed by $g$ is a subtree of $T$ as it is non-empty; therefore it contains the geodesic $[v,w]$. By acylindricity of the action, $[v,w]$ contains no paths of length longer than $k$; hence it is path-connected by Lemma \ref{boundeddiameter}.
\end{proof} 
\end{clly}

\begin{prop}\label{abelianactingacyl}
Let a profinite group $G$ act acylindrically on a profinite tree $T$. Let $A$ be a closed abelian subgroup of $G$. Then either $A\iso\Z[\pi]$ for some set of primes $\pi$ or $A$ fixes some vertex of $T$.
\begin{proof}
Let $S$ be a minimal invariant subtree for the action of $A$ on $T$. Then $A$ acts irreducibly on $S$, so by Proposition \ref{longgeods} either $S$ is a point (whence $A$ fixes a point of $T$) or $S$ contains paths of arbitrary length. By Proposition \ref{abelianactingontree}, if $A$ is not a projective group \Z[\pi] then the action is not faithful, so some non-trivial element of $A$ fixes $S$, hence fixes paths of arbitrary length. But this is impossible by acylindricity.
\end{proof}
\end{prop}

The following concept will be useful later. 
\begin{defn}\label{restnormdef}
Given a profinite group $G$ and a copy $C$ of $\hat \Z$ contained in it, the {\em restricted normaliser} of $C$ in $G$ is the closed subgroup
\[ {\cal  N}'_G(C)=\{g\in G\,|\, h^g = h \text{ or }h^g = h^{-1}\}\]
where $h$ is a generator of $C$. 
\end{defn}

Note that this is a closed subgroup of $G$, containing the centraliser as an index 1 or 2 subgroup. We deal with the restricted normaliser to avoid certain technicalities in later proofs; in particular the centraliser may not be of finite index in the full normaliser. There is a continuous homomorphism from the full normaliser ${\cal N}_G(C)$ to ${\rm Aut}(\hat \Z)$. The centraliser is the kernel of this map, and the reduced normaliser is the preimage of the unique order 2 subgroup ${\rm Aut}(\Z)\sbgp {\rm Aut}(\hat\Z)$ which acts non-trivially on each \Z[p].

\begin{prop}\label{overgroupsfixvertices}
Let a profinite group $G$ act $k$-acylindrically on a profinite tree $T$.  \begin{enumerate}[{(}i{)}]
\item Let $H\sbgp K$ be closed subgroups of $G$ with $[K:H]<\infty$, and suppose $H$ fixes a vertex $v$ of $T$. Then $K$ fixes some vertex of $T$. 
\item Let $C\iso\hat\Z$ be a subgroup of $G$ fixing some vertex $v$ of $T$. Then the centraliser ${\cal Z}_G(C)$ and the restricted normaliser ${\cal N}'_G(C)$ both fix some vertex of $T$.
\end{enumerate} 
\begin{proof}
First let us prove part (i). If $x\in K$ there exists $n\in\Z$ such that $x^n$ is in $H$; hence $x^n$ fixes $v$. Then by Proposition \ref{powersfix}, there is some vertex $u$ fixed by $x$. By Corollary \ref{samepathcpt}, $u$ and $v$ lie in the same path component of $T$, hence so do $u$ and $x\cdot v$. Thus $K$ acts on the path component $S$ of $v$ in $T$, which is an abstract tree. Every element of $K$ fixes some vertex of $S$; hence by Section I.6.5, Proposition 26 of \cite{Serre}, every finite set of elements of $K$ has a common fixed point in $S$, hence in $T$. Thus every finite intersection of the trees $T^x$ for $x\in K$ is non-empty; by compactness of $T$, their intersection $\bigcap T^x=T^K$ is non-empty. That is, $K$ fixes a vertex $w$ of $T$. 

Now consider part (ii). The reduced normaliser has the centraliser as an index 2 subgroup, so by (i) it suffices to consider the centraliser of $C$. Let $c$ be a generator of $C$, and let $x\in{\cal Z}_G(C)$. We will show that $x$ must fix a vertex of $T$ in the same path component of $T$ as $v$; the rest of the proof proceeds as for part (i). 

Let $A=\overline{<\!x,c\!>}$. This is abelian so by Proposition \ref{abelianactingacyl} either $A$ fixes a vertex $u$ of $T$ or $A$ is a projective group \Z[\pi]. If it is projective, let $a$ generate $A$. Then $c=a^\lambda$ for some $\lambda\in\hat\Z$. Because $C=\overline{<\!c\!>}\iso\hat\Z$, it follows that $\lambda$ is not a zero-divisor in $\hat\Z$. So by Proposition \ref{powersfix}, there is some vertex $u$ of $T$ fixed by $a$, hence by $A$. In either case, $A$ fixes a vertex $u$; since $c$ fixes $u$ and $v$, they lie in the same path component of $T$ and we are done. 

Note that in both parts of the Proposition, the vertex $w$ fixed by the subgroup in question is in the same path-component of $T$ as $v$, and moreover will be joined to it by a path of length less than $k$.
\end{proof}
\end{prop}

\begin{rmk}
The following profinite analogue of the above-quoted result in \cite{Serre} seems plausible, and would allow a strengthening of the above result. Let a profinite group $G$ act on a profinite tree $T$, and suppose that every element of $G$ fixes a vertex of $T$. Then $G$ fixes a vertex of $T$. However the author does not know if this result is true.
\end{rmk}

\section{Graphs of profinite groups}\label{secGofGs}
The theory of profinite graphs of groups can be defined for general profinite graphs $X$; we shall only consider finite graphs $X$ here as this considerably simplifies the theory and is sufficient for our needs. First we recall the notion of free profinite product, as developed in Section 9.1 of \cite{RZ00} and Chapter 4 of \cite{RZup}.
\begin{defn}
Given profinite groups $G_1,\ldots,G_n$ a free profinite product of the $G_i$ consists of a profinite group $H$ and morphisms $\phi_i:G_i\to H$ and which is universal with respect to this property. That is, for any other profinite group $K$ and morphisms $\psi_i:G_i\to K$ there is a unique map $f:H\to K$ such that $f\circ\phi_i=\psi_i$. 
\end{defn}
The free profinite product exists and is unique, and will be denoted
\[ G_1\amalg\ldots\amalg G_n\text{ or } \coprod_{i=1}^n G_i\]
Free profinite products are generally quite well-behaved, for instance for discrete groups $K_1, K_2$ we have
\[ \hat K_1\amalg\hat K_2 = \widehat{K_1\ast K_2} \]
Free products are a special case of a graph of groups in which all edge groups are trivial. We now move to the general definition.
\begin{defn}
A {\em finite graph of profinite groups} $(X, {\cal G}_\bullet)$ consists of a finite graph $X$, a profinite group ${\cal G}_x$ for each $x\in X$, and two (continuous) monomorphisms $\bdy_i:{\cal G}_x\to {\cal G}_{d_i(x)}$ for $i=0,1$ which are the identity when $x\in V(X)$. We will often suppress the graph $X$ and refer to `the graph of groups $\cal G$'.
\end{defn}
\begin{defn}
Given a finite graph of profinite groups $(X, {\cal G}_\bullet)$, choose a maximal subtree $Y$ of $X$. A profinite fundamental group of the graph of groups with respect to $Y$ consists of a profinite group $H$, and a map
\[ \phi : \coprod_{x\in X} {\cal G}_x \amalg \coprod_{e\in E(X)} \overline{<\! t_e\!>} \to H \]
such that    
\[ \phi(t_e) = 1 \text{ for all }e\in E(Y)\] and 
\[\phi( t_e^{-1} \bdy_0(g) t_e ) = \phi(\bdy_1(g)) \text{ for all } e\in E(X), g\in {\cal G}_e\]
and with $(H,\phi)$ universal with these properties. The profinite group $H$ will be denoted $\Pi_1(X,\cal G)$.
\end{defn}
Note that in the category of discrete groups this is precisely the same as the classical definition as a certain presentation. The group so defined exists and is independent of the maximal subtree $Y$ (see Section 5.2 of \cite{RZup}). 

In the classical Bass-Serre theory, a graph of discrete groups $(X,\cal G)$ gives rise to a fundamental group $\pi_1(X,\cal G)$ and an action on a certain tree $T$ whose vertices are cosets of the images $\phi({\cal G}_v)$ of the vertex groups in $\pi_1(X,\cal G)$ and whose edge groups are cosets of the edge groups. Putting a suitable topology and graph structure on the corresponding objects in the profinite world and proving that the result is a profinite tree, is rather more involved than the classical theory; however the conclusion is much the same. We collate the various results into the following theorem. 
\begin{theorem}[See Section 5.3 of \cite{RZup}]
Let $(X,\cal G)$ be a finite graph of profinite groups. Then there exists an (essentially unique) profinite tree $S(\cal G)$, called the {\em standard graph} of $\cal G$, on which $\Pi = \Pi_1(X,\cal G)$ acts with the following properties.  Set $\Pi(x) = {\rm im}({\cal G}_x\to \Pi)$.
\begin{itemize}
\item The quotient graph $\Pi\backslash S({\cal G})$ is isomorphic to $X$.
\item The stabiliser of a point $s\in S(\cal G)$ is a conjugate of $\Pi(\zeta(s))$ in $\Pi$, where $\zeta:S({\cal G})\to X$ is the quotient map.
\end{itemize}
\end{theorem}
Conversely (see Section 5.4 of \cite{RZup}) an action of a profinite group on a profinite tree with quotient a finite graph gives rise to a decomposition as a finite graphs of profinite groups. However no analogous result holds when the quotient graph is infinite.

In the classical theory one tacitly identifies each ${\cal G}_x$ with its image in the fundamental group $\pi_1(X,\cal G)$ of a graph of groups. In general in the world of profinite groups the maps $\phi:{\cal G}_x\to \Pi_1(X,\cal G)$ may not be injective, even for simple cases such as amalgamated free products. We call a graph of groups {\em injective} if all the maps $\phi$ are in fact injections. 

Let $(X,\cal G)$ be a finite graph of abstract groups. We can then form a finite graph of profinite groups $(X,\hat{\cal G})$ by taking the profinite completion of each vertex and edge group of $\cal G$. We have not yet addressed whether the `functors' 
\[ (X, {\cal G}) \to (X, \hat{\cal G}) \to \Pi_1(X, \hat{\cal G}) \]
and 
\[ (X, {\cal G}) \to \pi_1(X, {\cal G}) \to \widehat{\pi_1(X, {\cal G})}\]
on a graph of discrete groups $(X,\cal G)$ yield the same result; that is, whether the order in which we take profinite completions and fundamental groups of graphs of groups matters. In general the two procedures do not give the same answer; we require some additional separability properties. 
\begin{defn}
A graph of discrete groups $(X,\cal G)$ is {\em efficient} if $\pi_1(X,\cal G)$ is residually finite, each group ${\cal G}_x$ is closed in the profinite topology on $\pi_1(X,\cal G)$, and $\pi_1(X,\cal G)$ induces the full profinite topology on each ${\cal G}_x$.
\end{defn}
\begin{theorem}[Exercise 9.2.7 of \cite{RZ00}]
Let $(X,\cal G)$ be an efficient finite graph of discrete groups. Then $(X,\cal\hat G)$ is an injective graph of profinite groups and \[\widehat{\pi_1(X, {\cal G})}\iso\Pi_1(X, \hat{\cal G})\]
\end{theorem}
In our case of interest, the above-quoted Theorem \ref{JSJefficient} of Wilton and Zalesskii may be rephrased as `the JSJ decomposition of a 3-manifold group is efficient', and the profinite completion of our graph manifold group acts in a well-controlled fashion on a profinite tree.

\section{The JSJ decomposition}\label{secJSJ}
\begin{defn}
Let $M$ be an aspherical 3-manifold with JSJ decomposition $(X,M_\bullet)$. The {\em Seifert graph} of $M$ is the full subgraph of $X$ spanned by those vertices whose associated 3-manifold is a \SFS. It will be denoted \SSG.
\end{defn}
In this section we analyse the JSJ decomposition of an aspherical manifold, and show that `the Seifert-fibred part' is a profinite invariant. Specifically we prove:
\begin{theorem}\label{decompfixed}
Let $M$, $N$ be closed aspherical 3-manifolds with JSJ decompositions $(X,M_\bullet)$, $(Y,N_\bullet)$ respectively. Assume that there is an isomorphism $\Phi:\widehat{\pi_1 M}\to \widehat{\pi_1 N}$. Then there is an isomorphism $\varphi:\SSG\iso\SSG[Y]$ such that $\widehat{\pi_1 M_x}\iso \widehat{\pi_1 N_{\varphi(x)}}$ for every $x\in\SSG$.

If in addition $M$, $N$ are graph manifolds then, after possibly performing an automorphism of $\widehat{\pi_1 N}$, the isomorphism $\Phi$ induces an isomorphism of JSJ decompositions in the following sense:
\begin{itemize}
\item there is a graph isomorphism $\varphi:X\to Y$;
\item $\Phi$ restricts to an isomorphism $\widehat{\pi_1 M_x}\to \widehat{\pi_1 N_{\varphi(x)}}$ for every $x\in V(X)\cup E(X)$.
\end{itemize}
\end{theorem}

The proof will consist of an analysis of the $\hat\Z{}^2$-subgroups of $G=\widehat{\pi_1 M}$ together with the normalisers of their cyclic subgroups. We maintain the above notations for the rest of the section. Additionally let $S(\cal G)$ be the standard graph of a graph of profinite groups $(X,{\cal G}_\bullet)$, and let $\zeta:S({\cal G})\to X$ be the projection. Let $Z_v$ denote the canonical fibre subgroup of a vertex stabiliser $G_v$ which is a copy of the profinite completion of some \SFS{} group. By abuse of terminology, we will refer to vertices of $S({\cal G})$ as major, minor or hyperbolic when the corresponding vertex stabiliser is the profinite completion of a major or minor \SFS{} group or a cusped hyperbolic 3-manifold.

We will first show that the action on $S({\cal G})$ is 4-acylindrical.
\begin{rmk}
Wilton and Zalesskii use this fact in their paper \cite{WZ10} for graph manifolds, as well as in \cite{WZ14} and \cite{HWZ12} more generally. Their proof does however contain a gap. Specifically, their version of Lemma \ref{orbbdy} below only allows for conjugating elements $g_i$ in the original group \ofg[O], rather than its profinite completion. There is a similar problem in the hyperbolic pieces, which we deal with in Lemma \ref{malnormalcusps} below.
\end{rmk}
\begin{lem}\label{orbbdy}
Let $O$ be a hyperbolic 2-orbifold, and let $l_1, l_2$ be curves representing components of $\bdy O$. Let $B=\widehat{\ofg[O]}$, and let $C_i$ be the closure in $B$ of $\pi_1 l_i$. Then for $g_i\in B$, either $C_1^{g_1}\cap C_2^{g_2}=1$ or $l_1=l_2$ and $g_2 g_1^{-1}\in C_1$.
\end{lem}
\begin{proof}
By conjugating by $g_1^{-1}$ we may assume that $g_1=1$; drop the subscript on $g_2=g$. Note that $C_1\cap C_2^g$ is torsion-free, so it is sufficient to pass to a finite index subgroup $N$ and show that $C_1\cap C_2^g \cap N = 1$. 

Because $O$ is hyperbolic, it has some finite-sheeted regular cover $O'$ with more than two boundary components; then given any pair of boundary components, \ofg[O'] has a decomposition as a free product of cyclic groups, among which are the two boundary components. Let $N$ be the corresponding finite index normal subgroup of $B$. Note that for some set $\{h_i\}$ of coset representatives of $N\cap \ofg[O]$ in \ofg[O] (which give coset representatives of $N$ in $B$), each $C_2 ^ {h_i}$ is the closure of the fundamental group of a component of $\bdy O'$; so set $C'_2=C_2^{h_i}$ where $g=h_i n$ for some $n\in N$. Furthermore, if two boundary components of $O$ are covered by the same boundary component $O'$, then they must have been the same boundary component of $O$; that is, if $C_1\cap N=C'_2\cap N$, then $C_1=C'_2$. 

Now the intersections of $C_1, C'_2$ with $N$ are free factors; that is, \[N=(C_1\cap N)\amalg (C'_2\cap N)\amalg F\]
where $F$ is a free product of cyclic groups (unless $C_1=C'_2$, when $N=(C_1\cap N)\amalg F$). Let $T$ be the standard graph for this free product decomposition of $N$. Then $C_1\cap N=N_v, C'_2\cap N = N_w$ for vertices $v,w\in T$. The action on $T$ is 0-acylindrical because all edge stabilisers are trivial; so for $n\in N$, the intersection 
\[C_1\cap {C'}_2^n\cap N= N_v\cap N_{n^{-1}\cdot w}\] can only be non-trivial if $v=n^{-1}\cdot w$, so that $C_1\cap N=C'_2\cap N$ (hence $C_1=C'_2$) and $n\in C_1$.

We have reduced to the case where $C_2^{h_i}=C_1$. The intersection of two distinct peripheral subgroups of $\ofg[O]$ is trivial, and peripheral subgroups coincide with their normalisers in \ofg[O]. But \ofg[O] is virtually free, so by Lemma 3.6 of \cite{RZ96}, the intersection of their closures in $B$ is also trivial, so if $C_2^{h_i}\cap C_2\neq 1$, then $C_2^{h_i}=C_2$, and $h_i$ normalises $C_2$; hence it normalises the intersection with \ofg[O], and so $h_i\in C_2=C_1$ as required.
\end{proof}

\begin{lem}[Proposition 5.4 of \cite{WZ10}]\label{bdyintersect}
Let $L$ be a major \SFS{} with boundary, and let $O$ be its base orbifold. Let $H=\widehat{\pi_1 L}$, and $Z$ be the canonical fibre subgroup of $H$. Let $D_1$, $D_2$ be peripheral subgroups of $H$; that is, conjugates in $H$ of the closure of peripheral subgroups of $\pi_1 L$. Then $D_1\cap D_2=Z$ unless $D_1=D_2$.
\end{lem}
\begin{proof}
Certainly $Z$ is contained in the intersection $D_1\cap D_2$; if $D_1\cap D_2$ is strictly larger than $Z$, then the images $C_1,C_2$ of $D_1, D_2$ in $B=H/Z$ intersect non-trivially. Hence by the previous lemma $C_1=C_2$, hence $D_1=D_2$.
\end{proof}

\begin{lem}\label{edgegrps}
Let $e=[v,w]$ be an edge of $S({\cal G})$, where $v,w$ are Seifert fibred. Then $Z_v\times Z_w\sbgp[f] G_e$, and so $Z_v\cap Z_w =1$.
\end{lem}
\begin{proof}
After a conjugation in $G$, we may assume that $e$ is an edge in the standard graph of the abstract fundamental group $\pi_1 M$, i.e. $G_e$ is the closure in $G$ of a peripheral subgroup of some $\pi_1 M_v$. Because $\widehat{H_1\times H_2}\iso \hat H_1\times \hat H_2$ for groups $H_1, H_2$, the result now follows from the corresponding result in the fundamental group $\pi_1 M$; the canonical fibre subgroups are distinct direct factors of the edge group, which is therefore contains their product as a finite-index subgroup.
\end{proof}

\begin{lem}\label{malnormalcusps}
Let $L$ be a hyperbolic 3-manifold with toroidal boundary. Let $H=\widehat{\pi_1 L}$, and let $D_1$, $D_2$ be peripheral subgroups of $L$; that is, conjugates in $H$ of the closure of maximal peripheral subgroups of $\pi_1 L$. Then $D_1\cap D_2=1$ unless $D_1=D_2$ and, moreover, each $D_i$ is malnormal in $H$.
\end{lem}
\begin{proof}
Choose a basepoint for $L$ and let $P_1,\ldots, P_n$ be the fundamental groups of the boundary components of $L$ with this basepoint. These form a malnormal family of subgroups of $\pi_1 L$ and we wish to show that their closures are a malnormal family of subgroups of $H$. Suppose $g\in H$ is such that $\hat P_i\cap \hat P_j{}^g \neq 1$ and suppose that $i\neq j$ or $g\notin \hat P_i$. In the latter case let $q:\pi_1 L\to Q$ be a map to a finite group such that under its extension $\bar q:H\to Q$ to $H$, the image of $g$ does not lie in the image of $\hat P_i$; otherwise take $q$ to be the map to the trivial group. 

By Thurston's hyperbolic Dehn surgery theorem (see \cite{BP12}) we may choose slopes $p_k$ on the $P_k$ such that Dehn filling along each slope gives a closed hyperbolic 3-manifold $N$; moreover, we may choose such slopes with enough freedom to ensure that the image of $\hat P_i\cap \hat P_j{}^g$ is infinite in $H / \overline{{\ll}p_1,\ldots,p_n{\gg}}$. The images of the $P_k$ in this hyperbolic 3-manifold group $\pi_1 N$ are a malnormal family of subgroups. 

Now consider $K_0 = {\rm ker} (q)\cap {\ll}p_1,\ldots,p_n{\gg}$ and its closure in $H$. Note that $\pi_1 L/K_0$ is a hyperbolic virtually special group; indeed on the finite-index subgroup ${\rm ker}(q)$ we have just Dehn filled to get a closed hyperbolic manifold. Furthermore the images of the $P_k$ in $\pi_1 L/K_0$ are an almost malnormal family of subgroups. Then by Corollary 3.2 of \cite{WZ14}, their closures (i.e{.} the images of the $\hat P_k$ in $H/\bar K_0$) are an almost malnormal family of subgroups of $H/\bar K_0$. But since the maps to $\widehat{\pi_1 N}$ and the map $\bar q$ both factor through the map $H\to H/\bar K_0$, we have that the images of $\hat P_i$ and $\hat P_j^g$ intersect in an infinite subgroup, but $i\neq j$ and the image of $g$ does not lie in the image of $\hat P_i$. This contradiction completes the proof.
\end{proof}

\begin{lem}\label{onlycusps}
Let $L$ be a hyperbolic 3-manifold with toroidal boundary. Let $H=\widehat{\pi_1 L}$ and let $A$ be a subgroup of $H$ isomorphic to $\hat\Z{}^2$. Then $A$ is conjugate into a peripheral subgroup of $H$.
\end{lem}
\begin{proof}
This follows from the proof of Theorem 9.3 of \cite{WZ14}.
\end{proof}
\begin{prop}\label{JSJacyl}
The action of $G=\widehat{\pi_1 M}$ on the standard graph $S(\cal G)$ is 4-acylindrical.
\end{prop}
\begin{proof}
Take a path of length 5 consisting of edges $e_0, \ldots, e_4$ joining vertices $v_0,\ldots, v_5$. Let $M_i$ be the manifold $M_{\zeta(v_i)}$, where $\zeta:S({\cal G})\to X$ is the projection. If any of $M_1,\ldots, M_4$ is hyperbolic then the intersection of the two adjacent edge groups is trivial by Lemma \ref{malnormalcusps}. So assume all these $M_i$ are \SFS{}s. There are three cases to consider:
\begin{enumerate}[{Case} 1]
\item Suppose both $M_1$, $M_2$ are major \SFS{}s. Then by Lemma \ref{bdyintersect}, $G_{e_0}\cap G_{e_1}=Z_{v_1}$ and $G_{e_1}\cap G_{e_2}=Z_{v_2}$; but $Z_{v_1}\cap Z_{v_2}$ is trivial. So $\bigcap_{i=0}^2 G_{e_i}$ is trivial.
\item Suppose $M_1$ is a major \SFS{} and $M_2$ is a minor \SFS{}. Let $g$ be an element of $G_{v_2}\smallsetminus G_{e_1}$. Then $v_3=g\cdot v_1$ is major and $Z_{v_3} = Z_{v_1}^{g^{-1}}$. Then acting by $g$ sends $e_1$ to $e_2$ and fixes $v_1$; hence the intersection of all four edge groups, $Z_{v_1}\cap G_{e_2} \cap Z_{v_3}$, is a normal subgroup of $G_2$. Moreover, as the intersection of two direct factors of $G_{e_2}$, it is trivial or a maximal copy of $\hat\Z$ in $G_{e_2}$. Hence it is either trivial or is one of the two fibre subgroups of $G_{v_2}$. But the latter case is ruled out as neither of these fibre subgroups intersects $Z_{v_1}$ or $Z_{v_3}$ non-trivially. 
\item Suppose $v_1$ is minor. Then $v_2$ is major. If $v_3$ is major, then the argument of Case 1 applies. If $v_3$ is minor, then $v_4$ is major and relabelling $i\mapsto 5-i$ we are back in Case 2.\qedhere
\end{enumerate}
\end{proof}

\begin{clly}\label{JSJabel}
Let $(X,M_\bullet)$ be the JSJ decomposition of a graph manifold $M$, and let $G_\bullet = \widehat{\pi_1 M_\bullet}$. Let $A$ be an abelian subgroup of $G=\widehat{\pi_1 M}$. Then $A\iso\Z[\pi]$ for some set of primes or $A$ fixes some vertex of $S(\cal G)$. 
\end{clly}
\begin{proof}
Apply Proposition \ref{abelianactingacyl}.
\end{proof}

Having located the $\hat \Z{}^2$ subgroups, we proceed to distinguish those making up the edge groups between Seifert fibred pieces from those `internal' to the vertex group in which they are contained or adjacent to a hyperbolic piece. This will be accomplished using the normalisers and centralisers of cyclic subgroups. Recall that the centraliser of a subgroup $H\leq G$ is denoted ${\cal Z}_G(H)$, and that ${\cal N}'_G(C)$ denotes the restricted normaliser of a procyclic subgroup as defined in Definition \ref{restnormdef}.

\begin{defn}
A {\em non-pathological torus} in $G$ is a copy $A\leq G$ of $\hat\Z{}^2$, not contained in any larger copy of $\hat\Z{}^2$, with the following property. For every conjugate $A^g$ of $A$ in $G$, either $A\cap A^g=1$, $A\cap A^g$ is a subgroup of $\hat\Z$, or $A=A^g$.
\end{defn}
\begin{defn}
A procyclic subgroup $C\iso\hat\Z$ of $G$ is {\em major fibre-like} if:
\begin{itemize}
\item $C$ is a direct factor of some non-pathological torus of $G$;
\item ${\cal Z}_G(C)$ is not virtually abelian; and
\item $C$ is maximal with these properties.
\end{itemize} 
\end{defn}

The following result can be deduced from Lemma \ref{bdyintersect}; however the following proof, being much more elementary, merits inclusion.
\begin{prop}\label{bdycptsaremax}
Let $O$ be a hyperbolic 2-orbifold, and let $c$ be an element of $\ofg(O)$ representing a boundary component of $O$. Let $B=\widehat{\ofg[O]}$. Then the closed subgroup $C\leq\widehat{\ofg[O]}$ generated by $c$ is not contained in any strictly larger $\hat\Z$-subgroup of $B$. Hence any abelian subgroup of $B$ containing $C$ is $C$ itself.
\end{prop} 
\begin{proof}
Note that the quotient of $\hat \Z$ by any proper subgroup $C\iso\hat\Z$ is a direct product of finite groups, at least one of which is non-trivial; it follows that $C$ is contained in some sub-$\hat\Z$ with index a prime $p$. Hence it suffices to show that $c$ cannot be written as a $p^{\rm th}$ power $x^p$. If $c$ has this property in some quotient of $B$, then it also has this property in $B$ itself. Thus it suffices to find, for each $p$, a finite quotient of \ofg[O] in which the image of $c$ is not a $p^{\rm th}$ power. We split into cases based on the topological type of $O$. Furthermore, by passing to a quotient, it suffices to deal with the cases where $c$ is the only boundary component of $O$.
\begin{enumerate}[{Case} 1]
\item Suppose first that $O$ either is orientable of genus at least 1 or that it is non-orientable with at least three projective plane summands; so that $O$ has a punctured torus as a boundary-connected summand $O=({\mathbb T}^2\smallsetminus\ast)\#_\bdy O'$ , and on passing to a quotient we may assume that $O$ is a once-punctured torus, and $c=[a,b]$ where $a,b$ are free generators of the free group \ofg[O].

Suppose $p\neq 2$. Consider the mod-$p$ Heisenberg group 
\[H_p=\left\{ \begin{pmatrix}1 & \ast&\ast \\ 0 &1&\ast\\0&0&1\end{pmatrix} \right\} \leq {\rm SL}_3(\Z/p)\]
and map to it by \[a\mapsto \begin{pmatrix}1 & 1&0 \\ 0 &1&0\\0&0&1\end{pmatrix}, b\mapsto \begin{pmatrix}1 & 0&0 \\ 0 &1&1\\0&0&1\end{pmatrix}\] so that $c=[a,b]$ maps to  \[\begin{pmatrix}1 & 0&1 \\ 0 &1&0\\0&0&1\end{pmatrix}\]
Now in the Heisenberg group, the formula for an $n^{\rm th}$ power is
\[  \begin{pmatrix}1 & x&z \\ 0 &1&y\\0&0&1\end{pmatrix}^n = \begin{pmatrix}1 & nx&nz+ (1+2+\cdots +(n-1))xy \\ 0 &1&ny\\0&0&1\end{pmatrix} \]
so that in particular all $p^{\rm th}$ powers vanish, noting that $p$ is odd so divides $1+\cdots+(p-1)$. So the image of $c$ cannot possibly be a $p^{\rm th}$ power.

If $p=2$, instead map to the mod-4 Heisenberg group $H_4$ by the same formulae; then all squares have the form 
\[  \begin{pmatrix}1 & x&z \\ 0 &1&y\\0&0&1\end{pmatrix}^2 = \begin{pmatrix}1 & 2x&2z+ xy \\ 0 &1&2y\\0&0&1\end{pmatrix} \]
If this were to equal to image of $c$, then $x$ and $y$ would have to be even, so that $2x+xy$ would be even, and therefore not 1, a contradiction. So $c$ cannot be a square in $B$ either.

\item If $O$ is a punctured Klein bottle, with possibly some cone points, then after factoring out the cone points we have that \ofg[O] is a free group on two generators $a,b$ with $c=a^2b^2$. If $p\neq 2$ simply map to $\Z/p$ by $a\mapsto 1, b\mapsto 1$, so that $c\mapsto 4$ is not a multiple of $p$. If $p=2$ then map to the mod-4 Heisenberg group $H_4$ as above; this time the image of $c$ is
    \[\begin{pmatrix}1 & 2&0 \\ 0 &1&2\\0&0&1\end{pmatrix}\]
  If this were to be a square
  \[  \begin{pmatrix}1 & x&z \\ 0 &1&y\\0&0&1\end{pmatrix}^2 = \begin{pmatrix}1 & 2x&2z+ xy \\ 0 &1&2y\\0&0&1\end{pmatrix} =\begin{pmatrix}1 & 2&0 \\ 0 &1&2\\0&0&1\end{pmatrix}\]
  then $x,y$ would be odd, hence so would be $2z+xy$ which is thus non-zero. So the image of $c$ is not a square.
  
\item The remaining cases are either discs with at least two cone points or M\"obius bands with at least one cone point. After factoring out any excess cone points, we may assume \ofg[O] is either $(\Z/m)\ast(\Z/n)$ or $(\Z/m)\ast\Z$, and $c$ is expressed either as $ab$ or $ab^2$ in these generators. Consider the kernels $K$ of the maps to $(\Z/m)\times(\Z/n)$ or $\Z/m$ respectively. Some power $c^k$ of $c$ is then a boundary component of the cover corresponding to $K$, which is torsion-free hence yields one of the `high genus' cases above; hence $c^k$ is not a $p^{\rm th}$ power in $K$. If $p$ is coprime to $m,n$ then $x^p\in K$ if and only if $x\in K$, so $c^k$ is not a $p^{\rm th}$ power in $B$ either (hence neither is $c$ itself). Finally if $p$ divides one of $m$, $n$ (without loss of generality, if $p$ divides $m$), then mapping to the first factor sends $c$ to 1, which is not divisible by $p$ modulo $m$. 
\end{enumerate}
Finally note that since $B$ is a free product of cyclic groups, maximal abelian subgroups are (finite or infinite) procyclic. The proof is now complete.   
\end{proof}

\begin{prop}
If $v$ is a major vertex of $S({\cal G})$, then ${\cal Z}_G(Z_v)\leq G_v$ and ${\cal N}'_G(Z_v)=G_v$.
\end{prop}
\begin{proof}
The centraliser of $Z_v$ in $G$ contains the non-abelian group ${\cal Z}_{G_v}(Z_v)$, an index 1 or 2 subgroup of $G_v$. By Proposition \ref{overgroupsfixvertices}, the centraliser of $Z_v$ in $G$ is contained in a vertex group $G_w$; any two distinct vertex groups intersect in at most $\hat\Z{}^2$, so $v=w$ and the centraliser of $Z_v$ in $G$ is equal to the centraliser of $Z_v$ in $G_v$. Similarly for the reduced normaliser, noting that all of $G_v$ is contained in ${\cal N}'_G(Z_v)$.
\end{proof}

\begin{prop}
If $e$ is an edge of $S({\cal G})$, then $G_e$ is a non-pathological torus of $G$. 
\end{prop}
\begin{proof}
Suppose first that $e=[v,w]$ where $v$ is a major vertex. First note that $G_e$ is a maximal copy of $\hat\Z{}^2$.  For if $A=G_e$ were contained in a larger copy $A'$ of $\hat\Z{}^2$, then $A'$ would centralise $Z_v$, hence be contained in $G_v$. The image of $A'$ in $G_v/Z_v$ would be abelian, containing a copy of $\hat\Z$. By Proposition \ref{abelianactingacyl}, maximal abelian subgroups of $G_v/Z_v$ are finite or projective as it is a free product of procyclic groups; so the image of $A'$ is a copy of $\hat\Z$ properly containing a boundary component. By Proposition \ref{bdycptsaremax}, this is impossible. So these edge groups are maximal copies of $\hat\Z{}^2$. 

Now suppose $e=[v,w]$ where $v$ is hyperbolic and $w$ hyperbolic or minor, the only cases where $e$ has no major endpoints. Assume $A$ is contained in some strictly larger copy $A'$ of $\hat\Z{}^2$. Then $A'$ is contained in some vertex group $G_u$. The geodesic $[u,v]$ thus has stabiliser containing $A$. If $[u,v]$ has one edge then we are done as by Lemma \ref{malnormalcusps} all edge groups are maximal in adjacent vertex groups. If $[u,v]$ has more than two edges, recall that the peripheral subgroups of $G_v$ are malnormal by Lemma \ref{malnormalcusps}; from which we deduce that $[u,v]$ has at most two edges, those adjacent to the minor vertex $w$, and that $A'$ is contained in these edge groups. Thus $A=A'$ and again $A$ is maximal.

Suppose $g\in G$ and that $G_e\cap G^g_e$ is not a subgroup of $\hat\Z$. By the arguments in Proposition \ref{JSJacyl}, an intersection of two edge groups is at most $\hat\Z$ unless the edges are equal or are the two edges incident to a minor vertex $w$. So either $g^{-1}\cdot e=e$, whence $g\in G_e$ and $G_e^g=G_e$, or $g\in G_w$ and $G_e$ is normal in $G_w$ so that again $G_e\cap G_e^g=G_e$.
\end{proof}
\begin{prop}
If $v$ is a major vertex of $S({\cal G})$, then $Z_v$ is a major fibre-like subgroup of $G$.
\end{prop}
\begin{proof}
Note that $Z_v$ is a direct factor of an edge group $A=G_e$ ($e=[v,w]\in S({\cal G})$) which is a non-pathological torus, and has centraliser which is not virtually abelian. It remains to show that $Z_v$ is maximal with these properties; but again, if it were contained in a larger procyclic subgroup $C$ contained in a copy $A'$ of $\hat\Z{}^2$, this $A'$ would centralise $Z_v$, hence lie in $G_v$. The image of $A'$ in $G_v/Z_v$ would then be infinite abelian, hence projective; but killing a non-maximal copy $Z_v$ of $\hat\Z$ inside $A'$ would introduce torsion. Hence $Z_v$ is major fibre-like.
\end{proof}
\begin{prop}
Let $C\iso\hat\Z$ be a subgroup of $G$. If $C$ is major fibre-like then $C=Z_v$ for some major vertex $v$. If $C$ is contained in an edge group $G_e$ then either $C\leq Z_v$ for some (possibly minor) vertex $v$ or ${\cal Z}_G(C)\iso\hat\Z{}^2$.
\end{prop}
\begin{proof}
If $C$ is contained in an edge group, it suffices to replace $C$ with a maximal copy of $\hat\Z$ such that $C\leq\hat\Z\leq G_e$. Let $A$ be a non-pathological torus containing $C$ as a direct factor, choosing $A=G_e$ when $C$ is contained in an edge group. Note that $A$ is contained in a vertex group by Corollary \ref{JSJabel}.

Then $C$ and hence its centraliser lie in a vertex group $G_v$. If the centraliser is not just $A$, then $v$ is unique. If $v$ is hyperbolic then $A$ is a peripheral subgroup by Lemma \ref{onlycusps}; then if $x\in{\cal Z}_G(C)$, then $A\cap A^x\neq 1$ whence $x\in A$ by malnormality. If $v$ is minor, so that $A$ is an index 2 subgroup of $G_v$, the result is easy; so suppose $v$ is major. Note that, as a maximal abelian subgroup of $G_v/Z_v$ is procyclic, $Z_v$ is a direct factor of $A$, so that $C\times Z_v$ is a finite-index subgroup of $A$ unless $C=Z_v$. So suppose $C\neq Z_v$; we will show that $C$ is not major fibre-like, and that if $C$ was in a boundary component, its centraliser is exactly $A$. 

Consider first the index 1 or 2 subgroup ${\cal Z}_G(Z_v)=G'_v\leq G_v$ in which $Z_v$ is central. If $x\in G'_v\cap{\cal Z}_G(C)$, then $x$ commutes with both generators of both $C$ and $Z_v$; so $A^x\cap A$ is at least the finite index subgroup $C\times Z_v$. Because $A$ is non-pathological, $A^x=A$. Any action on $A\iso \hat\Z{}^2$ which is trivial on a finite-index subgroup of $A$ is trivial on all of $A$; so $\overline{<\!x,A\!>}$ is abelian. As before, taking the quotient gives an infinite abelian subgroup of $G_v/Z_v$, which is thus procyclic. Hence $\overline{<\!x,A\!>}$ is a copy of $\hat\Z{}^2$; by maximality this is $A$, hence $x\in A$ and $G'_v\cap{\cal Z}_G(C)=A$. So ${\cal Z}_G(C)$ is virtually abelian (with index 1 or 2), and $C$ is not major fibre-like.

Let $q:G_v\to G_v/Z_v$ be the quotient map. If $C$ is a direct factor of a boundary component $A=G_e$ other than $Z_v$, then the image of $C$ in $B=G_v/Z_v$ generates a finite-index subgroup of a peripheral subgroup $D=q(G_e)$ of the base orbifold. If $x\in{\cal Z}_G(C)$, then from above either $x\in G_e$ or $x^2\in G_e$.  Then $q(x)$ commutes with the finite-index subgroup of $D$ generated by $C$; so $D^q(x)$ intersects $D$ non-trivially, and by Proposition \ref{bdyintersect} we find $q(x)\in D$ so that $x\in q^{-1}(D)=G_e$.
\end{proof}
This last proposition shows that the property of being a canonical fibre subgroup $Z_v$ of a major vertex may defined intrinsically, as a `major fibre-like' subgroup. We will use this to show invariance of the JSJ decomposition.  
\begin{proof}[Proof of Theorem \ref{decompfixed}]
Let $M$, $N$ be aspherical 3-manifolds with JSJ decompositions $(X, M_\bullet)$, $(Y, N_\bullet)$, let $(X,{\cal G}_\bullet)$, $(Y, {\cal H}_\bullet)$ be the corresponding graphs of profinite groups, and let $S(\cal G)$, $S(\cal H)$ be the standard graphs for these graphs of profinite groups. Suppose there exists an isomorphism $\Phi:\widehat{\pi_1 M}=G\to H=\widehat{\pi_1 N}$. Let $\zeta:S({\cal G})\to X$ be the projection.

Let $A$ be any maximal copy of $\hat\Z{}^2$ in $G$. Then $A$ is contained in a vertex group, and in the centraliser of any of its cyclic subgroups. In particular, if it has two major fibre-like subgroups, then it is contained in two distinct major vertex groups, hence $A$ is some edge group $G_e$ and $\zeta(e)\in\SSG$. Conversely, if $G_e$ is any edge group where $\zeta(e)\in\SSG$, it has two major fibre-like subgroups; if $e=[v,w]$ where $v,w$ are major, then $Z_v,Z_w$ are major fibre-like subgroups of $G_e$; if $w$ is minor, then it has another adjacent vertex $v'$ with $G_{[v,w]}=G_{[v',w]}$, and so $Z_v$, $Z_{v'}$ are major fibre-like subgroups of $G_e$. They are distinct, otherwise as in the proof of Proposition \ref{JSJacyl} they coincide with a fibre subgroup of $G_w$, giving a contradiction. Furthermore, the intersection of any three major vertex groups is at most cyclic, so $G_e$ cannot contain three distinct major fibre-like subgroups.

Now construct an (unoriented) abstract graph $\Gamma$ as follows. The vertices of $\Gamma$ are the major fibre-like subgroups $Z_v$, and the edges are those maximal $\hat\Z{}^2$ subgroups containing two major fibre-like subgroups, i.e. the edge groups $G_e$ with $\zeta(e)\in\SSG$. The incidence maps are defined by inclusion. This incidence relation is not quite the same as incidence in $S(\cal G)$; two major vertex groups separated in $S(\cal G)$ by a minor vertex are now adjacent in $\Gamma$. We now rectify this. All maximal cyclic subgroups of an edge group $G_e$ have either have centraliser $\hat\Z{}^2$ or are $Z_v$ for some major or minor vertex $v$; so for each edge group $G_e$ with a third maximal procyclic subgroup with centraliser larger than $\hat\Z{}^2$, subdivide the corresponding edge of $\Gamma$ to get a new graph $\Gamma'$ with a vertex representing the minor vertex group whose canonical fibre subgroup is contained in $G_e$. Clearly $\Gamma'$ is isomorphic to $\zeta^{-1}(\SSG)$ as an abstract graph; the $G$-action on $\Gamma'$ induced from $S({\cal G})$ is determined by conjugation of the $Z_v$. On the other hand, the graph $\Gamma'$ and $G$-action so constructed are invariants of the group, so the isomorphism $\Phi:G\iso H$ yields an equivariant isomorphism of $\Gamma'$ with the corresponding object for $H$; hence  the quotient graphs \SSG{} and \SSG[Y] are isomorphic as claimed. Furthermore corresponding vertex stabilisers (i.e. the profinite completions of the Seifert fibred pieces of $M$ and $N$) are isomorphic. This completes the first part of the theorem. 

Now assume $M$, $N$ are graph manifolds. Now $\Gamma'\iso S(\cal G)$ as an abstract graph, hence we get an equivariant isomorphism $\Psi$ between $S({\cal G})$ and $S({\cal H})$, in the sense that 
\[\Psi(g\cdot v)= \Phi(g)\cdot \Psi(v)\]
for all $g\in G, v\in V(S({\cal G}))$. Note that this descends to an isomorphism 
\[ X = G\backslash S({\cal G}) \iso H\backslash S({\cal H}) = Y\]

We now check that the morphism of graphs thus constructed is in fact continuous, hence an isomorphism in the category of profinite graphs. For by Lemma \ref{Gdecomp}, $S({\cal G})$ is the inverse limit of its quotients by finite index normal subgroups of $G$. However, for each such subgroup $N$, we know that $\Psi$ induces a natural morphism of (abstract) graphs
\[ N\backslash S({\cal G}) \iso \Phi(N)\backslash S({\cal H}) \]
But these graphs, being finite covers of $X$ and $Y$ respectively (with covering groups $G/N, H/\Phi(N)$ respectively) are finite, hence these morphisms are continuous; that is, we have isomorphisms of inverse systems
\[ S({\cal G}) \iso \varprojlim N\backslash S({\cal G}) \iso\varprojlim\Phi(N)\backslash S({\cal H})\iso S({\cal H}) \]
so that our morphism $\Psi$ is indeed continuous.

Note that by Proposition \ref{overgroupsfixvertices}, the restricted normaliser of each $Z_v$ is contained in a vertex group; as it contains $G_v$ which is not contained in any edge group, we have that ${\cal N}'_G(Z_v)=G_v$. We now have an equivariant isomorphism $\Psi:S({\cal G})\to S({\cal H})$ of profinite graphs such that 
\[ G_v \iso \Phi(G_v) = \Phi({\cal N}'_G(Z_v)) = {\cal N}'_H(\Phi(Z_v))={\cal N}'_H(Z_{\Psi(v)}) = H_{\Psi(v)}\]
for $v\in V(S({\cal G}))$ and with each edge group the intersection of the adjacent vertex groups. This descends to an isomorphism
\[ X = G\backslash S({\cal G}) \iso H\backslash S({\cal H}) = Y\]
such that corresponding vertices and edges of $X$ and $Y$ have isomorphic associated groups; that is, we have an isomorphism ${\cal G}\iso {\cal H}$ of graphs of groups.

The fundamental group of a graph of profinite groups is well-defined independently of any choice of maximal subtree of $T$ and section $T\to S(\cal G)$; so that, following an automorphism of $H$, we may assume that the isomorphism $\Phi$ sends each vertex group of $\cal G$ to the corresponding vertex group of $\cal H$. See Section 5.2 of \cite{RZup} for details. 
\end{proof}

\section{Graph manifolds versus mixed manifolds}\label{secGvM}
The results of the previous section are seen to give very good information about graph manifolds; however picking up the precise nature of any hyperbolic pieces that may exist seems rather more subtle. In this section we show that the profinite completion does detect the presence of a hyperbolic piece. This can be seen as an extension of \cite{WZ14} in that we now know that the profinite completion determines which geometries arise in the geometric decomposition of a 3-manifold. 

\begin{theorem}
Let $M$ be a  (closed) mixed or totally hyperbolic manifold and $N$ be a graph manifold. Then $\pi_1 M$ and $\pi_1 N$ do not have isomorphic profinite completions. 
\end{theorem}
\begin{proof}
Let $M, N$ have JSJ decompositions $(X, M_\bullet)$, $(Y, N_\bullet)$ and assume $\widehat{\pi_1 M}\iso \widehat{\pi_1 N}$. Then by Theorem \ref{decompfixed}, $\SSG\iso \SSG[Y]=Y$ which is connected as $N$ is a graph manifold. Furthermore any finite-index cover of a graph manifold is a graph manifold and any finite-index cover of a mixed or totally hyperbolic manifold is mixed or totally hyperbolic. Hence any finite-index cover of $M$ has connected (and non-empty) Seifert graph. We will show that this is impossible. Let $M_1$ be a hyperbolic piece of $M$ adjacent to \SSG. Take a boundary torus $T$ of $M_1$ lying between $M_1$ and a major Seifert fibred piece of $M$. (If $M$ has only minor Seifert fibred pieces, it has a double cover with empty Seifert graph, which is forbidden).

Note that some finite-sheeted cover of $M_1$ has more than one boundary component which projects to $T$. The JSJ decomposition of $M$ is efficient by Theorem A of \cite{WZ10}; therefore some finite-sheeted cover $M'$ of $M$ induces a (possibly deeper) finite-sheeted cover $M'_1$ of $M_1$, which will still have more than one boundary component projecting to $T$. One such boundary torus $T'$ is now non-separating in $M'$. For all preimages of $T$ are adjacent to a Seifert fibred piece of $M'$, so if $T'$ were separating, the Seifert graph of $M'$ would be disconnected, which is forbidden. Cut along $T'$, take two copies of the resulting 3-manifold, and glue these together to get a degree two cover $\tilde M$ of $M'$. Removing the two copies of $M'_1$ from $\tilde M$ gives a disconnected manifold, each of whose components contains a Seifert fibred piece, so the Seifert graph of $\tilde M$ is disconnected as required.
\end{proof}

\section{Totally hyperbolic manifolds}\label{secPureHyp}
Theorem \ref{decompfixed} does not give any information about those manifolds whose JSJ decomposition has no Seifert fibred pieces at all; we shall call such manifolds `totally hyperbolic'. In this section we show that the analysis in Section \ref{secJSJ} does allow us to deduce some limited information about the JSJ decomposition of these manifolds, even without a way to detect what the vertex groups are. Specifically we will prove the following theorem.
\begin{theorem}\label{purehyp}
Let $M$, $N$ be totally hyperbolic manifolds with $G=\widehat{\pi_1 M}\iso \widehat{\pi_1 N}$ and with JSJ decompositions $(X, M_\bullet)$, $(Y, N_\bullet)$. Then the graphs $X$ and $Y$ have equal numbers of vertices and edges and equal first Betti numbers. 
\end{theorem}
\begin{proof}
Note that by Proposition \ref{abelianactingacyl} and Lemma \ref{onlycusps}, the maximal copies of $\hat\Z{}^2$ in $G$ are precisely the conjugates of completions of the JSJ tori $T$ of $M$. Thus immediately the profinite completion determines $|E(X)|$ as the number of such conjugacy classes. The action of conjugation on the homology group $H_2(G;\hat\Z)$ being trivial, each such conjugacy class gives an element of $H_2(G;\hat\Z)$, the image of a generator of $H_2(T;\Z)$ under the maps
\[ H_2(T;\Z)\to H_2(M;\Z)=H_2(\pi_1 M;\Z)\to H_2(G;\hat\Z)\]
This element $\xi_T$ is well-defined up to multiplication by an invertible element of $\hat \Z$. Furthermore it is either primitive or zero since the classes in $H_2(M;\Z)$ have this property and
\[ H_2(G;\hat\Z) \iso \varprojlim H_2(G;\Z/n)\iso \varprojlim H_2(\pi_1 M;\Z/n) \iso \hat\Z{}^{b_2(M)} \] 
where the middle isomorphism holds since 3-manifold groups are good in the sense of Serre (this is proved by a combination of the Virtual Fibring Theorem of Agol and Theorem B of Wilton-Zalesskii \cite{WZ10}). See Proposition 6.5.7 of \cite{RZ00} for the inverse limit.

Notice that the first Betti number of the graph $X$ is equal to the maximal number of edges that can be removed without disconnecting the graph. On the level of the 3-manifold $M$, this equals the size of a maximal collection of JSJ tori that are together non-separating. A collection of these tori is non-separating if and only if the corresponding homology classes $\eta_{T_1},\ldots, \eta_{T_k}\in H_2(M;\Z)$ generate a subgroup of $H_2(M;\Z)$ of rank $k$. The closure of a subgroup $\Z^r$ of $H_2(M;\Z)\iso \Z^n$ in $H_2(G;\hat\Z)\iso\hat\Z{}^n$ is isomorphic to $\hat\Z{}^r$. This closure is precisely the closed subgroup generated by the $\zeta_{T_i}$, hence the profinite completion detects whether a collection of tori is non-separating, hence the Betti numbers of $X$ and $Y$ must be equal. This completes the proof.
\end{proof} 

\section{The pro-$p$ JSJ decomposition}\label{secproP}
Let $p$ be a prime. We have focussed thus far on the full profinite completion of a graph manifold group because of the good separability results which hold for all graph manifolds, and because of the good profinite rigidity properties of \SFS{}s. The pro-$p$ topology on a graph manifold or \SFS{} is in general rather poorly behaved; indeed most are not even residually $p$. In this section we will note that when the pro-$p$ topology is well-behaved, the arguments of the previous section still suffice to prove a pro-$p$ version of Theorem \ref{decompfixed}. First let us discuss the pro-$p$ topologies on \SFS{} groups. The following lemmas and proposition are well-known but are included for completeness.
\begin{lem}[Gr\"unberg \cite{Gru57}, Lemma 1.5]
Let $G$ be a group and suppose $H$ is a subnormal subgroup of $G$ with index a power of $p$. If $H$ is residually $p$, so is $G$.
\end{lem}
\begin{lem}
Let $O$ be an orientable orbifold with non-positive Euler characteristic and such that each cone point of $O$ has order a power of $p$. Then $O$ has a subregular cover of degree a power of $p$ which is a surface. Hence $\ofg{O}$ is residually $p$.
\end{lem}
\begin{proof}
We will suppose that $O$ has no boundary, as this is the more difficult case. We will construct a normal subgroup of $\ofg{O}$ which is in some sense simpler than $\ofg{O}$; iterating this process will terminate in a subregular subgroup which is torsion-free, hence a surface group. A presentation for $\ofg{O}$ is
\[ \big< a_1,\ldots, a_m, u_1,v_1,\ldots, u_g,v_g\,\big|\, a_i^{p^{n_i}} = 1, a_1\cdots a_m[u_1,v_1]\cdots [u_g,v_g]=1\big>\]
where the cone points have order $p^{n_1}\leq p^{n_2}\leq\cdots\leq p^{n_m}$. Consider the abelian $p$-group
\[ A = \Z/(p^{n_1})\times\cdots \times\Z/(p^{n_m}) \,/\, (1,1,\ldots,1)  \]
and the map $\phi:\ofg{O}\to A$ sending each generator $a_i$ to the $i^{\rm th}$ coordinate vector and sending $u_j,v_j$ to 0. Let $K$ be the kernel of $\phi$. Then if there are at least two cone points of maximal order $p^{n_m}$, no power of any $a_i$ (other than the identity) lies in the kernel of $\phi$ since no multiple of $(1,1,\ldots,1)$ has precisely one non-zero coordinate; hence $K$ is torsion-free as required. 

If there is only one cone point, then since the Euler characteristic is non-positive, $g>0$ and we may pass to a regular cover of degree $p$ with $p$ cone points.

 Otherwise, if $n_{m-1}<n_m$ and there are at least two cone points, then any torsion elements in $K$ are conjugates of $a_m^{p^j}$ for $j>0$, so that their order is at most $p^{n_m-1}$; thus the order of the maximal cone point of the cover corresponding to $K$ is strictly smaller than that of $O$. Iterating this process will thus eliminate all cone points of $O$, giving the required subnormal subgroup.
\end{proof}
\begin{prop}\label{respSFS}
Let $p$ be a prime. Let $M$ be a \SFS{} which is not of geometry $\sph{3}$ or $\sph{2}\times\R$. Then $\pi_1 M$ is residually $p$ if and only if all exceptional fibres of $M$ have order a power of $p$, and $M$ has orientable base orbifold when $p\neq2$. That is, $M$ has residually $p$ fundamental group precisely when its base orbifold $O$ is $\Z/p$-orientable and has residually $p$ fundamental group. Moreover, when this holds the following sequence is exact:
\[1\to\Z[p]\to\proP{\pi_1 M}\to\proP{\ofg{O}}\to 1 \]
where \Z[p] is generated by a regular fibre of $M$.
\end{prop}
\begin{proof}
$(\Leftarrow)$ By the above, such base orbifolds have a subregular cover of degree a power of $p$ which is an orientable surface. Using this subregular cover we need only verify that if the base orbifold of $M$ is an orientable surface $\Sigma$ then the \SFS{} group is residually $p$. Since surface groups and free groups are $p$-good, for each $n$ the map \[H^2(\proP{\pi_1 \Sigma};\Z/p^n)\to H^2(\pi_1 \Sigma;\Z/p^n)\] is an isomorphism; hence the central extension
\[1\to \Z/p^n\to \pi_1(M)/<\!h^{p^n}\!\!> \to \pi_1 \Sigma\to 1\]
(where $h$ represents a regular fibre of $M$) is the pullback of a central extension of \proP{\pi_1 \Sigma} by $\Z/p^n$; whence each quotient $\pi_1(M)/<\!h^{p^n}\!\!>$ is residually $p$. This proves that $\pi_1 M$ is residually $p$ and, moreover, that the canonical fibre subgroup is closed in the pro-$p$ topology, so that we have the exact sequence as claimed. 

$(\Rightarrow)$ Recall that a subgroup of a residually $p$ group is residually $p$. Suppose first that $O$ is non-orientable and $p\neq 2$. Then the subgroup of $\pi_1 M$ generated by $g$ and $h$ is $\Z\rtimes\!<\!g\!>$ where the copy of \Z[] is the fibre subgroup $<\!h\!>$ and $g$ acts by inversion. Then $g^2$ acts trivially on \Z; but in any finite $p$-group quotient of $\pi_1 M$, the subgroup generated by $g^2$ contains $g$, so the image of $[g,h]$ vanishes in any $p$-group quotient; so $\pi_1 M$ is not residually $p$.

Now let $p$ be arbitrary and suppose some exceptional fibre $a$ has order $p^n m$ where $m\neq 1$ is coprime to $p$. Then if $b=a^{p^n}$, in any finite $p$-group quotient the image of $b$ is some power of the image of $b^m$, which is central; so $b$ is central in any finite $p$-group quotient of $\pi_1 M$. So if $\pi_1 M$ were residually $p$, the centre of $\pi_1 M$ would contain $b$, which it does not. So all exceptional fibres of $M$ have order $p^n$ for some $n$. 
\end{proof}
The pro-$p$ completion of course does not determine the \SFS{} quite as strongly as the profinite completion. However once the fundamental group is residually $p$ the techniques of \cite{BCR14} and \cite{Wilkes15}, combined with the above proposition, yield the following surprisingly strong results:
\begin{theorem}
Let $O_1$, $O_2$ be 2-orbifolds whose fundamental groups are residually $p$. If $\proP{\ofg{O_1}}\iso\proP{\ofg{O_2}}$ then $O_1\iso O_2$.
\end{theorem}
\begin{theorem}
Let $M_1$, $M_2$ be \SFS{}s whose fundamental groups are residually $p$. If $\proP{\pi_1 M_1}\iso\proP{\pi_1 M_2}$ then:
\begin{itemize}
\item $M_1$, $M_2$ have the same geometry;
\item $M_1$, $M_2$ have the same base orbifold; and
\item the Euler numbers of $M_1$, $M_2$ have equal $p$-adic norm.
\end{itemize}
\end{theorem}

Moving on to graph manifolds, we must define what constitutes a `nice' pro-$p$ topology for a graph manifold group. A primary tool we used in Section \ref{secJSJ} was the efficiency of the graph of groups representing a graph manifold group. There is a precisely analogous property in the pro-$p$ world:
\begin{defn}
A graph of discrete groups $(X,\cal G)$ is {\em $p$-efficient} if $\pi_1(X,\cal G)$ is residually $p$, each group ${\cal G}_x$ is closed in the pro-$p$ topology on $\pi_1(X,\cal G)$, and $\pi_1(X,\cal G)$ induces the full pro-$p$ topology on each ${\cal G}_x$.
\end{defn}
Needless to say, this property does not hold for the majority of graph manifolds; in particular all the Seifert fibred pieces must be of the form specified in Proposition \ref{respSFS}. However, the study of graph manifolds whose JSJ decomposition is $p$-efficient is by no means vacuous, as shown by the following theorem:
\begin{thmquote}[Aschenbrenner and Friedl \cite{AF13}, Proposition 5.2]
Let $M$ be a closed graph manifold and let $p$ be a prime. Then $M$ has some finite-sheeted cover with $p$-efficient JSJ decomposition.
\end{thmquote}
Notice that the profinite properties of graph manifolds used in Section \ref{secJSJ} to prove Theorem \ref{decompfixed} were:
\begin{itemize}
\item efficiency of the JSJ decomposition;
\item Lemma \ref{orbbdy} concerning intersections of boundary components; and
\item Proposition \ref{bdycptsaremax} concerning maximality of peripheral subgroups.
\end{itemize}
For $p$-efficient graph manifolds, observe that in the proofs of the above propositions it suffices to use only regular covers of order a power of $p$ and $p$-group quotients; and observe that the remainder of the results are applications of the above together with the theory of profinite groups acting on profinite trees, which works just as well (in fact better) in the category of pro-$p$ groups. Hence the same arguments prove the following pro-$p$ version of Theorem \ref{decompfixed}: 
\begin{theorem}\label{proPdecompfixed}
Let $M$, $N$ be closed graph manifolds with $p$-efficient JSJ decompositions $(X,M_\bullet)$, $(Y,N_\bullet)$ respectively. Assume that there is an isomorphism $\Phi:\proP{\pi_1 M}\to \proP{\pi_1 N}$. Then, after possibly performing an automorphism of $\proP{\pi_1 N}$, the isomorphism $\Phi$ induces an isomorphism of JSJ decompositions in the following sense:
\begin{itemize}
\item there is an (unoriented) graph isomorphism $\varphi:X\to Y$;
\item $\Phi$ restricts to an isomorphism $\proP{\pi_1 M_x}\to \proP{\pi_1 N_{\varphi(x)}}$ for every $x\in V(X)\cup E(X)$.
\end{itemize}
\end{theorem}

\section{Graph manifolds with isomorphic profinite completions}\label{secGMrigid}
We will now address the question of when two graph manifold groups can have isomorphic profinite completions. We will first restrict attention to those graph manifolds whose vertex groups have orientable base space. Let us recall from the standard theory of graph manifolds how one obtains numerical invariants of a graph manifold $M$. Let the JSJ decomposition be $(X, M_\bullet)$. In this section we shall adopt the convention (from Serre) that each `geometric edge' of a finite graph is a pair $\{e,\bar e\}$ of oriented edges, with $\bar e$ being the `reverse' of $e$. Fix presentations in the standard form
\[\left< a_1,\ldots, a_r, e_1, \ldots, e_s, u_1, v_1, \ldots, u_g, v_g, h\mid a_i^{p_i} h^{q_i}, h \text{ central}\, \right>\]
for each $M_x\, (x\in VX)$. This determines an ordered basis $\{h, e_i\}$ for the fundamental group of each boundary torus of $M_x$, where the final boundary component is described by 
\[e_0 = (a_1\cdots a_r e_1\cdots e_s [u_1, v_1]\cdots[u_g, v_g])^{-1}\]
Thus the gluing map along an edge $e$ takes the form of a matrix (acting on the left of a column vector) 
\[\begin{pmatrix}
\alpha(e) & \beta(e)\\ \gamma(e) & \delta(e)
\end{pmatrix} \] 
where $\gamma(e)$, the intersection number of the fibre of $d_0(e)$ with that of $d_1(e)$, is non-zero by the definition of a graph manifold. The number $\gamma(e)$ is well-defined up to a choice of orientation of the fibres of the two vertex groups. This matrix has determinant $-1$ from the requirement that the graph manifold be orientable. Once an orientation of the fibre and base are fixed, the number $\delta(e)$ becomes independent of the choice of presentation, modulo $\gamma(e)$ (changing these orientations multiplies the matrix by $-1$). The `section' of the fibre space determining $\delta(e)$ may be changed by a Dehn twist along an annulus which either joins two boundary components or joins the boundary component to an exceptional fibre. This however leaves the {\em total slope}
\[ \tau(x) = \sum_{e: d_0(e)=x} \frac{\delta(e)}{\gamma(e)}  - \sum \frac{q_i}{p_i} \]
of the space invariant. Note also that these quantities are all invariant under the conjugation action of the group on itself, i.e.\ it does not matter which conjugate of each edge group we consider.

\begin{theorem}\label{GMrigidor}
Let $M$, $N$ be closed orientable graph manifolds with JSJ decompositions $(X,M_\bullet)$, $(Y,N_\bullet)$, where all major vertex spaces have orientable base orbifold. Suppose $\widehat{\pi_1 M}\iso \widehat{\pi_1 N}$ and let $\phi\colon X\to Y$ be the induced graph isomorphism from Theorem \ref{decompfixed}.
\begin{enumerate}
\item If $X$ is not a bipartite graph, then $M$ is homeomorphic to $N$ via a homeomorphism inducing the graph isomorphism $\phi$.
\item Suppose $X$ is bipartite on vertex sets $R, B$.  Then there exists an element $\kappa\in\hat{\Z}{}^{\!\times}$ such that the following conditions hold, for some choices of orientations of the fibres of each major vertex group of $M$ or $N$:
\begin{enumerate}
\item[(a)] For all $e\in EX$, we have $\gamma(e)=\gamma(\phi(e))$
\item[(b)] For every vertex space of both $M$ and $N$ the total slope of that vertex is zero.
\item[(c1)] For every vertex $r\in R$, $(M_r, N_{\phi(r)})$ is a Hempel pair of scale factor $\kappa$ and for every edge $e$ with $d_0(e) = r$ we have $\delta(\phi(e)) \equiv \kappa \delta(e)$ modulo $\gamma(e)$.
\item[(c2)] For every vertex $b\in B$, $(M_b, N_{\phi(b)})$ is a Hempel pair of scale factor $\kappa^{-1}$ and for every edge $e$ with $d_0(e) = b$ we have $\delta(\phi(e)) \equiv \kappa^{-1} \delta(e)$ modulo $\gamma(e)$.
\end{enumerate}
\end{enumerate}  
\end{theorem}
\begin{rmk}
The conditions (a)-(c1) (or (c2)) in the above theorem have a rather neat interpretation in terms of the following object.
\begin{defn}
	Let $M=(X,M_\bullet)$ be a graph manifold with the given JSJ decomposition, and let $x\in VX$. Define the {\em filled vertex space} $\overline{M_x}$ of $M_x$ as follows. For every edge $e$ with $d_0(e)=x$, the fibre of $M_{d_1(x)}$ describes a meridian on the relevant boundary torus $T_e$ of $M_v$. There is a unique way to glue in a solid torus along $T_e$ with this meridian such that the fibring on $M_v$ extends over the solid torus. Gluing in a solid torus in this way gives a closed \SFS{} $\overline{M_x}$.   
\end{defn}
 Suppose that $(M_r, N_{\phi(r)})$ is a Hempel pair for scale factor $\kappa$, and for simplicity of notation suppose that $M_r$, $N_{\phi(r)}$ have the same base orbifold $O$. Now the filled vertex space $\overline{M_r}$ has an exceptional fibre with invariants $(\gamma(e), -\delta(e))$ for each boundary torus $e$ of $M$ and has Euler number $\tau(e)$. Perform the same operation on $N_{\phi(r)}$ to obtain $\overline{N_{\phi(r)}}$. Then (a) becomes the statement that $\overline{M_r}$, $\overline{N_{\phi(r)}}$ still have the same base orbifold $\overline{O}$, (b) states that both have Euler number zero, and (c1) states that they are still a Hempel pair!  
\end{rmk}
\begin{proof}
Let $\Phi\colon\widehat{\pi_1 M}\to \widehat{\pi_1 N}$ be an isomorphism. By (the proof of) Theorem \ref{decompfixed} we have an isomorphism $\phi\colon X\to Y$ and isomorphisms $\Phi_x\colon\widehat{\pi_1 M_x}\to \widehat{\pi_1 N_{\phi(x)}}$ for every $x\in X$. Choose some orientation of the fibres (i.e.\ generators of the fibre subgroups), giving an identification of the fibre subgroup with $\hat \Z$. Then by Theorem \ref{Wilkeswbdy} we may identify the base orbifolds of $M_x$ and $N_{\phi(x)}$ in such a way that $\Phi_x$ splits as an isomorphism of short exact sequences 
\[\begin{tikzcd}
1\ar{r} & \hat\Z\ar{r} \ar{d}{\cdot \lambda_x} & \widehat{\pi_1 M_x} \ar{r} \ar{d}{\Phi_x} & \widehat{\ofg[O_x]} \ar{r} \ar{d}{\psi_x} & 1 \\
1\ar{r} & \hat\Z\ar{r} & \widehat{\pi_1 N_{\phi(x)}} \ar{r} & \widehat{\ofg[O_x]} \ar{r} & 1
\end{tikzcd}\]
where $\lambda_x$ is some invertible element of $\hat\Z$ and $\psi_x$ is an automorphism of $B=\widehat{\ofg[O]}$ which sends each boundary component $e_i$ to a conjugate of $e_i^{\mu_x}$ and sends each cone-point $a_i$ to a conjugate of $a_{i}^{\mu_x}$, for some $\mu\in\hat\Z {}^{\!\times}$. That is, $(M_x, N_{\phi(x)})$ is a Hempel pair with scale factor $\lambda_x/\mu_x$. Note that, given an orientation on $M$ and of each fibre, the base orbifolds also inherit an orientation.

First consider a major vertex $x$ of $X$. As before, pick standard presentations of every major vertex group of $M$ and $N$. Then the maps $\Phi_x$ induce a map
\[\begin{pmatrix} \lambda_x & \rho_e\\ 0&\mu_x\end{pmatrix}\]
from $\widehat{\pi_1 M_e}$ to $\widehat{\pi_1 N_{\phi(e)}}$ for each edge $e$ with $d_0(e) = x$, where the $\rho_e$ may be specified as follows. Let $\pi_1 M_x$ have presentation 
\[\left< a_1,\ldots, a_r, e_1, \ldots, e_s, u_1, v_1, \ldots, u_g, v_g, h\mid a_i^{p_i} h^{q_i}, h \text{ central}\, \right>\] and $\pi_1 N_{\phi(x)}$ have presentation
\[\left< a_1,\ldots, a_r, e_1, \ldots, e_s, u_1, v_1, \ldots, u_g, v_g\mid {a}_i^{p'_i}= {h}^{q'_i}, h \text{ central}\right>\]
where we use the same letters to denote various generators using the identification of the base orbifolds. Then we have 
\[ \Phi_x(e_i) \sim {e_i}^{\mu_x}h^{\rho_{e_i}}\]
for each edge $e_i$, where $\sim$ denotes conjugacy. For the remaining edge $e_0 = (a_1\cdots a_r e_1\cdots e_s [u_1, v_1]\cdots[u_g, v_g])^{-1}$ we have of necessity
\[\Phi_x(e_0) \sim {e_0}^{\mu_x} h^{-(\rho_{e_1}+\cdots+\rho_{e_s}+\nu_1+\cdots\nu_r)}   \]
where the $\nu_i$ are the unique elements of $\hat\Z$ such that
\[\mu_x q'_i = \lambda_x q_i + \nu_i p_i \]
or, in terms of the map $\Phi_x$, given by $\Phi_x(a_i) \sim {a_i}^{\mu_x} {h}^{\nu_i}$.
These definitions of $\nu_i$ are equivalent due to the fact that $\Phi_x$ is an isomorphism. Hence we find 
\begin{equation} 
\sum_{d_0(e) = x} \rho_e  = \sum \lambda_x \frac{q_i}{p_i}  -\sum \mu_x\frac{q'_i}{p_i} \label{DTcondition} 
\end{equation}
an analogue of the classic restrictions on Dehn twists taking one \SFS{} presentation to another.

For minor pieces the isomorphism also induces a map on the single boundary component, this time represented by a diagonal matrix as the boundary component of the minor piece has a canonical basis (up to multiplication by $\pm 1$), where the second basis element is given by the unique maximal cyclic subgroup which is normal but not central. Equation \eqref{DTcondition} is still satisfied as all terms are zero.

Now, the fact that the gluing maps along any edge $e=\overrightarrow{xy}$ commute with $\Phi$ forces the equation 
\begin{equation}
\begin{pmatrix}
\alpha(\phi(e)) & \beta(\phi(e))\\
\gamma(\phi(e)) & \delta(\phi(e))
\end{pmatrix} = \begin{pmatrix}
\lambda_y & \rho_{\bar e} \\ 0&\mu_y
\end{pmatrix}  \begin{pmatrix}
\alpha(e) & \beta(e)\\
\gamma(e) & \delta(e)
\end{pmatrix}\begin{pmatrix}
\lambda_x &  \rho_e\\ 0&\mu_x
\end{pmatrix}^{-1}
\end{equation}
to hold. To simplify notation, write $\alpha'(e) = \alpha(\phi(e))$ and so on, and suppress the edge $e$ for the time being. In this notation we have
\begin{equation}
\begin{pmatrix}
\alpha' & \beta'\\
\gamma' & \delta'
\end{pmatrix}= \begin{pmatrix}
\lambda_x^{-1}\lambda_y\alpha+\lambda_x^{-1}\bar\rho\gamma & [\cdots]\\ \lambda^{-1}_x\mu_y \gamma & \mu_x^{-1}\mu_y\delta - \lambda_x^{-1}\mu_x^{-1}\mu_y\rho \gamma
\end{pmatrix} \label{matreqn}
\end{equation}
where the upper right entry is complicated and unimportant. The worried reader should note that both sides have determinant $-1$ and that the lower-left entry is not a zero-divisor in $\hat\Z$; these considerations determine the last entry.

Now we have $\lambda_x^{-1}\mu_y \gamma = \gamma'$. Since $\gamma,\gamma'$ are non-zero elements of $\Z$, this (by Lemma \ref{nonzerodiv}) forces $\lambda_x = \pm_e \mu_y$. Consideration of the determinant yields $\lambda_y = \pm_e \mu_x$; that is, $\pm_{\bar e} = \pm_e$.

Suppose that $X$ admits a cycle of odd length. Then the above relations, tracked around this cycle, force $\lambda_x = \pm \mu_x$ for some (and hence, by connectedness and the previous paragraph, every) vertex $x$. Here the $\pm$ sign is independent of $x$, so after changing the orientation on $M$ we may assume it is 1. But then $(M_x, N_{\phi(x)})$ are a Hempel pair of stretch factor $\lambda_x/\mu_x= 1$; that is, they are isomorphic Seifert fibre spaces. It follows that if we can find integers $r_e\in\Z$ for every edge such that
\[\begin{pmatrix}
\alpha(\phi(e)) & \beta(\phi(e))\\
\gamma(\phi(e)) & \delta(\phi(e))
\end{pmatrix} = \begin{pmatrix}
1 & r_{\bar e} \\ 0&1
\end{pmatrix}  \begin{pmatrix}
\alpha(e) & \beta(e)\\
\gamma(e) & \delta(e)
\end{pmatrix}\begin{pmatrix}
1 &  r_e\\ 0&1
\end{pmatrix}^{-1}\]
and with these $r_e$ realisable by Dehn twists- that is, with
\[ \sum_{d_0(e) = x} r_e = \sum\frac{q_i}{p_i} -\sum \frac{q'_i}{p'_i} \]
then $\pi_1 M$, $\pi_1 N$ are isomorphic. But since for each edge we have
\[ \delta - \mu_x^{-1} \rho \gamma = \delta'\]
the numbers  $r_e=\mu_x^{-1} \rho_e$ are in fact integers by Lemma \ref{nonzerodiv}. Equation \eqref{DTcondition} now states that these $r_e$ are actually realisable by Dehn twists. Hence $M$ and $N$ are homeomorphic (by a homeomorphism covering $\phi$) and we have proved the theorem in the case when $X$ is not bipartite.

Now take $X$ bipartite, on two sets $R$ and $B$. We will first take care of the signs $\pm_e$ by changing the fibre orientations of certain vertex groups of $N$. Pick some basepoint $x_0\in R$ and a maximal subtree $T$ of $X$ and, moving outward from the basepoint along $T$, change such fibre orientations as are required to force $\pm_e= +$ for every edge $e$ of $T$. Now define $\lambda = \lambda _{x_0}, \mu= \mu_{x_0}$. We now have 
\[ \lambda_r=\lambda, \mu_r=\mu, \lambda_b=\mu, \mu_b= \lambda\]
for all vertices $r\in R, b\in B$. For every remaining edge $e$ (with $d_0(e)=r\in R$ and $d_1(e)=b\in B$, say) we have $\lambda=\lambda_r=\pm_e \mu_b = \pm_e \lambda$, whence $\pm_e=+$ for all $e$.  

Now define $\kappa=\lambda/\mu$. For all edges $e=\overrightarrow{xy}$ we have $\gamma(\phi(e))= \lambda_x^{-1} \mu_y \gamma(e) = \gamma(e)$, so condition (a) of the theorem holds. By construction of $\kappa$ and the matrix equation \eqref{matreqn}, properties (c1) and (c2) also hold. Finally for any $x\in R$ we have
\begin{eqnarray*}
\tau(\phi(x))&=&\sum_{d_0(e)=x} \frac{\delta(\phi(e))}{\gamma(\phi(e))} - \sum \frac{q'_i}{p'_i}\\
 &=& \sum_{d_0(e)=x} \frac{\kappa\delta(e)-\mu_x^{-1}\rho_e \gamma(e)}{\gamma(e)} - \sum \frac{q'_i}{p'_i}\\
 &=& \sum_{d_0(e)=x} \frac{\kappa\delta(e)}{\gamma(e)} - \sum \kappa\frac{q_i}{p_i}= \kappa\tau(x)
\end{eqnarray*}
where the deduction of the last line uses equation \eqref{DTcondition}. A similar equation holds for vertices in $B$. Now (again by Lemma \ref{nonzerodiv}) this forces either $\kappa=\pm 1$ when, as in the non-bipartite case, $M$ and $N$ are isomorphic; or it forces $\tau(x) = \tau(\phi(x))=0$ for all $x$, i.e.\ we have condition (b) of the theorem. This completes the proof in this direction.
\end{proof}
The above arguments make it clear that, unless graph manifolds are to be rigid, we must be able to realise non-trivial values of the quantity `$\mu_x$', else $\kappa$ would have to be $\pm 1$. That is, we must be able to find exotic automorphisms of type $\mu$. The term `exotic automorphism' was defined in Definition \ref{exauto}. We recall it here for convenience:
\begin{defnrecall}
Let $O$ be an orientable 2-orbifold with boundary, with fundamental group
\[B= \left< a_1,\ldots, a_r, e_1, \ldots, e_s, u_1, v_1, \ldots, u_g,v_g\mid a_i^{p_i}\right>\]
where the boundary components of $O$ are represented by the conjugacy classes of the elements $e_1, \ldots, e_s$ together with
\[ e_0 = \left(a_1 \cdots a_r e_1\cdots e_s [u_1, v_1]\cdots [u_g,v_g]\right)^{-1} \]
Then an {\em exotic automorphism of $O$ of type $\mu$} is an automorphism $\psi\colon\hat B\to\hat B$ such that $\psi(a_i) \sim a_i^\mu$ and $\psi(e_i)\sim e_i^\mu$ for all $i$, where $\sim$ denotes conjugacy in $\hat B$.
Similarly, let $O'$ be a non-orientable 2-orbifold with boundary, with fundamental group 
\[B'= \left< a_1,\ldots, a_r, e_1, \ldots, e_s, u_1,\ldots, u_g\mid a_i^{p_i}\right>\]
where the boundary components of $O'$ are represented by the conjugacy classes of the elements $e_1, \ldots, e_s$ together with
\[ e_0 = \left(a_1 \cdots a_r e_1\cdots e_s  u_1^2\cdots u_g^2\right)^{-1} \]
Let $o\colon \hat B'\to \{\pm 1\}$ be the orientation homomorphism of $O'$. Let $\sigma_0, \ldots, \sigma_s\in\{\pm1\}$. Then an {\em exotic automorphism of $O'$ of type $\mu$ with signs $\sigma_0, \ldots, \sigma_s$} is an automorphism $\psi\colon\hat B'\to\hat B'$ such that $\psi(a_i) \sim a_i^\mu$ and $\psi(e_i)= (e_i^{\sigma_i\mu})^{g_i}$ for all $i$ where $o(g_i)=\sigma_i$ for all $i$.
\end{defnrecall} 

\begin{theorem}[\cite{belyui80}, see also Lemma 9 of \cite{GDJZ15}]\label{Belyi}
A sphere with three discs removed admits an exotic automorphism of type $\mu$ for every $\mu\in{\hat\Z}{}^{\!\times}$. 
\end{theorem}
We will now extend this to all 2-orbifolds with boundary, and then describe the automorphisms of the profinite completions of Seifert fibre space groups that we will use. 
\begin{prop}
For every $n$, a sphere with $n$ discs removed admits an exotic automorphism of type $\mu$, for every $\mu\in\hat\Z{}^{\!\times}$.
\end{prop}
\begin{proof}
The proof is by induction on $n$, starting at $n=3$. Suppose that $\psi_{n-1}$ is an exotic automorphism of an $(n-1)$-holed sphere $S_{n-1}$; by applying an inner automorphism we may assume that some boundary loop $e_0$ is sent by $\psi_{n-1}$ to precisely $e_0^\mu$. Let $\psi_3$ be an exotic automorphism of type $\mu$ of a 3-holed sphere $S_3$. Let the fundamental group of $S_3$ be generated by $a,b$; again we may assume that for the third boundary component $E=(ab)^{-1}$ we have $\psi_3(E)= E^\mu$.

Now gluing $e_0$ to $E$ produces an $n$-holed sphere $S_n$ with fundamental group
\[G=\left< \pi_1 S_{n-1}, \pi_1 S_3\mid E=e_0^{-1} \right> \] 
By construction, defining $\psi_n\colon\hat G\to \hat G$ by $\psi_{n-1}$ and $\psi_3$ on the two vertex groups of the amalgam gives a well-defined map from $\hat G$ to $\hat G$ to itself. This map is an isomorphism by the universal property of amalgamated free products (or because $\hat G$ is Hopfian). We now have an exotic automorphism $\psi_n$ as required.
\end{proof}
\begin{prop}\label{orbexauto}
Let $O$ be an orientable 2-orbifold with boundary. Then $O$ admits an exotic automorphism of type $\mu$ for any $\mu\in\hat\Z$. Moreover this automorphism may be induced by an automorphism of the orbifold $\mathring O$ obtained from $O$ by removing a small disc about each cone point.
\end{prop}
\begin{proof}
Let the fundamental group of $O$ have presentation \[B= \left< a_1,\ldots, a_r, e_1, \ldots, e_s, u_1, v_1, \ldots, u_g,v_g\mid a_i^{p_i}\right>\] with notation as in Definition \ref{exauto} and let $F$ be a discrete free group of rank $r+s+2g$ on a generating set \[\{a_1,\ldots, a_r, e_1, \ldots, e_s, v'_1, v_1, \ldots, v'_g,v_g\}\] realised as the fundamental group of an appropriately punctured sphere $S$. Let $\psi$ be an exotic automorphism of $S$ of type $\mu$ and let 
\[ \psi(v_i) = (v_i^\mu)^{g_i},\quad \psi(v'_i) = ({v'_i}^\mu)^{g'_i}\]
Let $G$ be the fundamental group of $\mathring O$ and write $G$ as an iterated HNN extension
\[\left< F, u_1,\ldots, u_g\mid v'_i = (v_i^{-1})^{u_i}\right>\]
Then we may extend $\psi$ over the iterated HNN extension $\hat G$ by setting 
\[\psi(u_i) = g_i^{-1} u_i g'_i \]
Note that the map $\psi$ so defined is actually a surjection, by the standard criterion for HNN extensions- we have a surjection on the vertex group and, on factoring out the vertex group, an epimorphism of free groups (in this case the identity). Since all finitely generated profinite groups are Hopfian, we do have an isomorphism witnessing the fact that $\mathring O$ admits an exotic automorphism of type $\mu$. Note that $u_i$ does not represent a boundary component of $O$, so it is of no concern that its conjugacy class is not preserved.

 Note that $\psi$ preserves the normal subgroup generated by the $a_i^{p_i}$, so descends to an automorphism of the quotient group of $\hat G$ by the normal subgroup generated by these elements. This quotient group is precisely the profinite completion of the fundamental group of $O$ as required.
\end{proof}
\begin{prop}\label{orbexautononor}
Let $O$ be a non-orientable 2-orbifold with $s+1$ boundary components. Let $\mu\in\hat\Z {}^{\!\times}$ and let $\sigma_0, \ldots, \sigma_s\in\{\pm1\}$. Then $O$ admits an exotic automorphism of type $\mu$ with signs $\sigma_0, \ldots, \sigma_s\in\{\pm1\}$. Moreover this automorphism may be induced by an automorphism of the orbifold $\mathring O$ obtained from $O$ by removing a small disc about each cone point.
\end{prop}
\begin{proof}
Let the fundamental group of $O'$ be
\[B'= \left< a_1,\ldots, a_r, e_1, \ldots, e_s, u_1,\ldots, u_g\mid a_i^{p_i}\right>\]
with notation as in Definition \ref{exauto}. Write $B'$ as an amalgamated free product 
\[\left< a_1,\ldots, a_r, e_1, \ldots, e_s, u_1, v_1,\ldots, u_g, v_g\mid a_i^{p_i}, u_i^2 = v_i \right>\]
Let $F$ be a free group on $r+s+g$ generators with generating set 
\[\{a_1,\ldots, a_r, e_1, \ldots, e_s, v_1, \ldots, v_g\}\]
considered as the fundamental group of a $(r+s+g+1)$-punctured sphere $S$. Take an exotic automorphism of $S$ of type $\mu$ and extend over the amalgam 
\[\hat F' = \left< a_1,\ldots, a_r, e_1, \ldots, e_s, u_1, v_1,\ldots, u_g, v_g\mid , u_i^2 = v_i \right> \] 
by sending each $u_i$ to $(u_i^\mu)^{g_i}$, where $\psi(v_i) = (v_i^\mu)^{g_i}$. Note that all conjugating elements $g_i$ lie in the completion of the fundamental group of the orientable surface $S$, which is contained in the kernel of the orientation homomorphism.

To introduce the signs $\sigma_i$ we may compose with automorphisms induced by automorphisms of the discrete group $F'$. Of course it suffices to change the signs one at a time. For instance to realise $\sigma_0 =-1$ we may compose with the map given by 
\[ v_g \mapsto (a_1\cdots a_r e_1\cdots e_s v_1^2\cdots v_{g-1}^2)^{-1}v_g^{-1} \]
and by the identity on all other generators. This sends 
\begin{eqnarray*}
 e_0 &=& (a_1\cdots a_r e_1\cdots e_s v_1^2\cdots v_{g}^2)^{-1}\\
 &\mapsto & \left( v_g^{-1}(a_1\cdots a_r e_1\cdots e_s v_1^2\cdots v_{g-1}^2)^{-1}v_g^{-1} \right)^{-1}\\
 &=& v_g a_1\cdots a_r e_1\cdots e_s v_1^2\cdots v_{g-1}^2 v_g\\
 &=& v_g e_0^{-1} v_g^{-1}
\end{eqnarray*} 
as required.
 Finally factoring out the relations $a_i^{p_i}$ gives the required exotic automorphism of $O'$.
\end{proof}
\begin{theorem}\label{SFSexmorph}
Let $M$, $N$ be orientable \SFS{}s. Suppose that $(M,N)$ is a Hempel pair with scale factor $\kappa=\lambda/\mu$ for some $\lambda,\mu\in\hat\Z{}^{\!\times}$. Then we may identify the base orbifolds of $M$ and $N$ with the same $O$ and choose `standard form' presentations 
\[G=\pi_1 M=\left< a_1,\ldots, a_r, e_1, \ldots, e_s, u_1, \ldots, h\mid a_i^{p_i} h^{q_i}, h^g=h^{o(g)}\, \right>\] 
and 
\[H=\pi_1 N=\left< a_1,\ldots, a_r, e_1, \ldots, e_s, u_1, \ldots, h\mid {a_i}^{p_i} h^{q'_i}, h^g=h^{o(g)} \right>\]
for the fundamental groups where $q'_i$ is congruent to $\lambda\mu^{-1} q_i$ modulo $p_i$ and where $o$ is the orientation homomorphism (see Definition \ref{deforcover} below).
 Let $\rho_0,\ldots,\rho_s$ be any elements of $\hat\Z$ such that 
\begin{equation}\label{DThypotheis}
\sum \rho_i = \lambda\sum  \frac{q_i}{p_i}-\mu\sum  \frac{q'_i}{p_i}
\end{equation} 
and let $\sigma_0,\ldots, \sigma_s\in\{\pm1\}$. Suppose that all $\sigma_i$ are $+1$ if the base orbifold is orientable. 

Let $B=\ofg[O]$. Then there exists an isomorphism of short exact sequences
\[\begin{tikzcd}
1\ar{r} & \hat\Z\ar{r} \ar{d}{\cdot \lambda} & \hat G \ar{r} \ar{d}{\Psi} & \hat B \ar{r} \ar{d}{\psi} & 1 \\
1\ar{r} & \hat\Z\ar{r} & \hat H \ar{r} & \hat B \ar{r} & 1
\end{tikzcd}\]
where $\psi$ is an exotic automorphism of type $\mu$ and where the map on each boundary component is described by
\[\Psi(h) = h^\lambda, \quad \Psi(e_i) = (e_i^{\sigma_i\mu})^{g_i}h^{\rho_i} \]
for some $g_i\in \hat H$ mapping to $\sigma_i$ under the orientation homomorphism on $\hat B$.
\end{theorem}
\begin{proof}
Define $\theta_i$ to be the unique element of $\hat\Z$ such that 
\[\mu q'_i = \lambda q_i + \theta_i p_i \]
and note that the hypothesis of the theorem forces
\[\sum \rho_i + \sum \theta_j = 0 \] 
Let $\tilde\psi$ be an exotic automorphism of $O\smallsetminus(\text{nbhd of cone points})$ of type $\mu$ as in Proposition \ref{orbexauto} or \ref{orbexautononor} such that 
\[\tilde\psi(a_i) = ({a_i}^\mu)^{f_i}, \quad \tilde\psi(e_i)=({e_i}^{\sigma_i \mu})^{g_i} \]
Let $\iota$ be the natural homomorphism to $\hat H$ from the free profinite group on the generators $\{a_i, e_i, u_i, (v_i)\}$. As in Propositions \ref{orbexauto} and \ref{orbexautononor} the map $\tilde\psi$ induces an exotic automorphism of $O$. Now define $\Psi\colon \hat G\to\hat H$ as follows:
\[\Psi(h)=h^\lambda,\quad\Psi(a_i)=\iota\tilde\psi(a_i){h}^{\theta_i}, \quad \Psi(e_i)=\iota\tilde\psi(e_i){h}^{\rho_i} \]  
and by $\iota\tilde\psi$ on the remaining generators. This map $\Psi$ is well-defined by the definition of the $\theta_i$ and is an isomorphism as it induces isomorphisms on both fibre and base. It has all the advertised properties by construction, except possibly the condition on $\Psi(e_0)$. For $O$ non-orientable we have
\begin{eqnarray*}
 \Psi(e_0) &=& \Psi(a_1\cdots a_r e_1\cdots e_s u_1^2\cdots u_g^2)^{-1}\\
 &=& \left(\iota\tilde\psi(a_1\cdots u_g^2) h^\wedge\big\{\sum \theta_i + \sum_{i\neq 0}\rho_i \big\}\right)^{-1}\\
 &=& \iota\tilde\psi(e_0) h^{\rho_0} = (e_0^{\sigma_0 \mu})^{g_0}h^{\rho_0}
\end{eqnarray*}
as required, noting that all elements $\iota\tilde\psi(u_i^2)$ et cetera commute with $h$. The case of orientable $O$ is similar.
\end{proof}

\begin{defn}\label{deforcover}
 If $M_v$ is a closed Seifert fibre space, there is an {\em orientation homomorphism} from $\widehat{\pi_1 M_v}$ to $\{\pm 1\}$ with kernel the centraliser of the canonical fibre subgroup. Let $M_v^{\rm or}$ be the {\em orientation cover}, i.e.\ the cover of $M_v$ corresponding to this subgroup.
 
 If $M$ is a graph manifold with JSJ decomposition $(X,M_\bullet)$, let $M^{\rm or}$ be the regular cover of $M$ induced by taking the orientation cover of each major vertex with non-orientable base orbifold. Denote the JSJ decomposition by $(X^{\rm or}, M^{\rm or}_\bullet)$. Note that, since every loop in $X$ may be realised by a loop in $M$ which lifts to $M^{\rm or}$, the graph $X$ is bipartite if and only if $X^{\rm or}$ is bipartite, and the bipartition on $X$ lifts to that on $X^{\rm or}$. Also note that $(\overline{M_v})^{\rm or} =\overline{(M_v^{\rm or})}$ where the bar denotes the filled vertex space.
\end{defn}	
Suppose that we have two graph manifolds $M$ and $N$ be graph manifolds with JSJ decompositions $(X, M_\bullet)$ and $(Y, N_\bullet)$, and a graph isomorphism $\phi\colon X\to Y$ taking vertices with non-orientable base orbifold to vertices with non-orientable base orbifold, and major vertices to major vertices. There is also a graph isomorphism $\phi^{\rm or}\colon X^{\rm or}\to Y^{\rm or}$ covering $\phi$. If there is a homeomorphism of $M^{\rm or}$ with $N^{\rm or}$ covering the map $\phi^{\rm or}$, then $M$ and $N$ are homeomorphic. For the fact that the graph isomorphism covers $\phi$ and that the different sorts of vertices match up correctly implies that the $\Z/2$ actions on each vertex space of $M^{\rm or}$ and $N^{\rm or}$ are matched up by this homeomorphism (or at least, one isotopic to it).

We will now move towards the final theorem governing the profinite completions of graph manifold groups. As in the discussion preceding Theorem \ref{GMrigidor}, choosing a generator of the fibre subgroup of each vertex group allows us to define a matrix representing the gluing maps on each edge $e$. Before this matrix was independent under conjugacy of the edge group. When base orbifolds can be non-orientable, all these invariants are independent of conjugacy only up to sign (unless both adjacent vertex groups have orientable base). The reader should not then be unduly surprised by the presence of sign indeterminacy in the following theorem- it is a consequence of ambiguities in the graph manifolds themselves, rather than anything mysterious concerning the profinite completions. Also note that the definition of `total slope' only involved the ratios $\delta/\gamma$ so is well-defined.

\begin{theorem}\label{GMrigid}
	Let $M$ and $N$ be graph manifolds with JSJ decompositions $(X, M_\bullet)$ and $(Y, N_\bullet)$ respectively.
	 \begin{enumerate}
		\item If $X$ is not bipartite, then $\pi_1 M$ is profinitely rigid.
		\item If $X$ is bipartite, on two sets $R$ and $B$, then $\pi_1 M$ and $\pi_1 N$ have isomorphic profinite completions if and only if, for some choices of generators of fibre subgroups, there is a graph isomorphism $\phi\colon X\to Y$ and some $\kappa\in\hat{\Z}{}^{\!\times}$ such that:
		\begin{enumerate}
\item[(a)] For each edge $e$ of $X$, $\gamma(\phi(e))=\pm_e \gamma(e)$, where the sign is positive if both end vertices of $e$ have orientable base.
\item[(b)] The total slope of every vertex space of $M$ or $N$ vanishes
\item[(c1)] If $d_0(e)=r\in R$, $\delta(\phi(e))=\pm_e \kappa \delta(e)$ modulo $\gamma(e)$, and $(M_r, N_{\phi(r)})$ is a Hempel pair of scale factor $\kappa$.
\item[(c2)] If $d_0(e)=b\in B$, $\delta(\phi(e))=\pm_e \kappa^{-1} \delta(e)$ modulo $\gamma(e)$, and $(M_b, N_{\phi(b)})$ is a Hempel pair of scale factor $\kappa^{-1}$. 
		\end{enumerate}
	\end{enumerate}
\end{theorem}
\begin{rmk}
	As in Theorem \ref{GMrigidor}, these conditions (a)-(c1) are almost equivalent to the requirement that the filled manifolds $\overline{M_r}$ and $\overline{N_{\phi(r)}}$ form a Hempel pair. The only difference is that forming the filled manifold forgets the orientation of the fibre of the adjacent manifold; hence it would not guarantee that the signs may be fixed as in the final part of (a). This is a necessary condition, as was seen in Theorem \ref{GMrigidor}.
\end{rmk}
\begin{proof}
	We will first deduce the `only if' direction from Theorem \ref{Gdecomp}. Let $\Psi\colon\widehat{\pi_1 M}\to \widehat{\pi_1 N}$ be an isomorphism and let $\phi\colon X\to Y$ be the induced graph isomorphism.  Since the isomorphism $\Psi$ preserves centralisers, there is an induced isomorphism  \[\Psi\colon\widehat{\pi_1 M^{\rm or}}\to \widehat{\pi_1 N^{\rm or}}\] and an induced isomorphism $\phi^{\rm or}\colon X^{\rm or}\to Y^{\rm or}$ which covers $\phi$.
	
	Suppose first that $X$ is not bipartite. Then neither is $X^{\rm or}$, so by Theorem \ref{GMrigidor} there is a homeomorphism $M^{\rm or}\to N^{\rm or}$ covering $\phi^{\rm or}$. By the comments following Definition \ref{deforcover}, it follows that $M$ is homeomorphic to $N$.
	
	Now suppose that $X$ is bipartite. We claim that the conclusion of Theorem \ref{GMrigidor} applied to $M^{\rm or}$ and $N^{\rm or}$ immediately forces the equations in (a)-(c2) to hold. For some choices of fibre orientations of the vertex spaces in the covers give the equations in the theorem statement for evey lift of an edge $e$ of $X$, and different lifts of $e$ have values of $\gamma, \delta$ et cetera equal to those of $e$, up to a choice of sign. Hence all the equations of the statement hold up to sign. It only remains to show that we may fix the signs on edges between two spaces of orientable base as in (a). Consider the graphs $A_M$ and $A^{\rm or}_M$ obtained from $X$ and $X^{\rm or}$ by removing all major vertices with non-orientable base, and all preimages of those vertices. Of course, $A_M$ contains all the edges we are concerned with. The components of $A_M^{\rm or}$ are homeomorphic copies of the components of $A$. Similarly define $A_N$ and $A_N^{\rm or}$. So choosing orientations on the fibres of $M^{\rm or}$ and $N^{\rm or}$  as in Theorem \ref{GMrigidor} gives a consistent choice of fibre orientations on the vertex spaces each component of $A_M$ or $A_N$ satisfying (a), simply by choosing some lift of the component to (for example) $A_M^{\rm or}$ and inheriting orientations from there. Hence the theorem is true in this case as well.
	
	Now let us turn our attention to the `if' direction of (2). Suppose that $M$ and $N$ are as in the conditions of that theorem. We will build a more or less explicit isomorphism of profinite groups, defined with respect to a presentation of $\pi_1 M$. We say `more or less' explicit as we cannot describe precisely the conjugating elements in exotic automorphisms of orbifolds. This unfortunately requires a small hurricane of notation.
	
	First let us describe the presentations that we will use. We will use primes to denote the invariants $\gamma(e), \delta(e)$ et cetera deriving coming from $N$ as opposed to $M$. That is, $\gamma'(e) = \gamma(\phi(e))$. We may as well identify $X$ and $Y$ using the isomorphism $\phi$. Choose a maximal subtree $T$ in $X$.  We may also identify the base orbifolds of $M_x$, $N_x$. Fix some presentation for each such orbifold. We will use the same letters to denote the generating sets of $\pi_1 M_x$, $\pi_1 N_x$ in a presentation coming from some presentation for the orbifold group, with $h_x$ denoting the fibre subgroup (with the choice of generator in the theorem statement), and with the letter $e$ denoting the element describing a meridian on the boundary torus if $x$ representing the edge $e$ (with $d_0(e)=x$). The meridian on the other vertex group adjacent to $e$ will then be denoted $\bar e$, being $e$ with the opposite orientation. Then the conditions of the theorem, and  the definitions of the invariants involved, specify the relations in $\pi_1 M$ and $\pi_1 N$ coming from each edge $e$. For instance, for an edge $e$ from $x$ to $y$ we have relations
	\[h_x^{(t_e)} = h_y^{\alpha(e)}{\bar e}^{\gamma(e)}, \quad e^{(t_e)} = h_y^{\beta(e)}{\bar e}^{\delta(e)}  \]
	in $\pi_1 M$, where $(t_e)$ is either the identity if $e\in T$ or a stable letter for the HNN extension over $e$ if $e\notin T$. Similarly in $\pi_1 N$ we have 
	\[h_x^{(t_e)} = h_y^{\alpha'(e)}{\bar e}^{\gamma'(e)}, \quad e^{(t_e)} = h_y^{\beta'(e)}{\bar e}^{\delta'(e)}  \]
	Now that we have fixed all the fibre orientations and a presentation, all these numbers become well-defined. To unify treatment of vertices in $R$ and $B$, define $\lambda_r = \kappa, \mu_r = 1$ for $r\in R$ and $\lambda_b = 1, \mu_b = \kappa$ for $b\in B$. We therefore also have a well-defined sign $\pm_e$ and well-defined $\rho(e)\in\hat{\Z}$ for each edge $e$ such that 
	\begin{equation} \gamma'=\pm_e \gamma, \quad \mu_x\delta'=\pm_e(\lambda_x \delta-\rho \gamma) \label{rhodef}\end{equation}
	when $d_0(e) = x$. The existence of $\rho$ is guaranteed by (c1) or (c2), and its uniqueness from the fact that $\gamma$ is not a zero-divisor in $\hat\Z$. Note that, by inverting the relations above, we find $\pm_{\bar e}=\pm_e$ and $\alpha(e) = -\delta(\bar e)$ so that 
	\begin{equation}\mu_y \alpha'=\pm_e(\lambda_y \alpha+\bar\rho \gamma)  \label{rhobardef}\end{equation}
	for $d_1(e)=y$, where $\bar{\rho}(e)= \rho(\bar e)$.
	
	Finally for each edge $e$ define $\sigma_e\in\{\pm 1\}$ as follows. If $x=d_0(e)$ has orientable base, set $\sigma_e=+1$. If $x\in R$ has non-orientable base, set $\sigma_e = \pm_e 1$. If $x\in B$ has non-orientable base, set $\sigma_e = \pm_e 1$ if $d_1(e)$ has orientable base, and $+1$ otherwise. Note that $\sigma_e \sigma_{\bar e} = \pm_e 1$.
	
	Now define an isomorphism $\Psi_x\colon \widehat{\pi_1 M_x}\to \widehat{\pi_1 N_x}$ using Theorem \ref{SFSexmorph} with input values $\lambda_x, \mu_x, \{\rho(e)\}, \{\sigma_e\}$. The condition (b) guarantees that the hypothesis \eqref{DThypotheis} of Theorem \ref{SFSexmorph} is satisfied because:
\begin{eqnarray*}
\sum \rho_i &=& \sum \frac{\lambda_x \delta(e)}{\gamma(e)} - \sum \frac{\pm_e\mu_x \delta'(e)}{\gamma(e)}\\
&=& \lambda_x \sum \frac{\delta(e)}{\gamma(e)} - \mu_x\sum \frac{\delta'(e)}{\gamma'(e)}\\
&=& \lambda_x \sum \frac{q_i}{p_i} - \mu_x\sum \frac{q'_i}{p'_i}
\end{eqnarray*}
	
	We are at long last in a position to build the promised isomorphism $\Omega$ from $\widehat{\pi_1 M}$ to $\widehat{\pi_1 N}$. First we will build a map defined on the vertex groups on the tree $T$, then deal with HNN extensions. Let $G_\bullet = \pi_1 M_\bullet$ and $H_\bullet = \pi_1 N_\bullet$. Choose some basepoint $t\in T$ and define $\Omega = \Psi_t$ on $\hat G_t$. Order the vertices of $X$ as $\{x_1,\ldots,x_n\}$ starting with $t$ such that each $\{x_1,\ldots,x_m\}$ spans a subtree $T_m$ of $T$. Suppose inductively that we have defined $\Omega$ coherently on the tree of groups $(T_m, \hat G_\bullet)$ such that $\Omega$ is defined on $\hat G_{x_i}$ by $\Omega(g) = \Psi_{x_i}(g)^{f_i}$ where $f_i$ is an element of $\hat H$ of the form 
	\[f_i = k_i k_{i-1}\cdots k_1 \] 
    where $k_j\in \hat H_j$ for $j\leq i$. There is an edge $e$ of $T$ with $d_0(e) = x = x_j\in \{x_1,\ldots,x_m\}$ and with $d_1(e)=y=x_{m+1}$. By construction of $\Psi_x$ there is $g_e\in \hat H_x$ such that $g_e$ evaluates to $\sigma_e$ under the orientation homomorphism on $\hat H_x$ and 
	\[\Psi_x(h_x)=h_x^{\lambda_x} ,\quad \Psi_x(e) = (e^{\sigma_e \mu_x})^{g_e}h_x^{\rho_e}= (e^{\sigma_e \mu_x}h_x^{\sigma_e\rho_e})^{g_e}\]
	Similarly we have $g_{\bar e}\in \hat H_y$ evaluating to $\sigma_{\bar e}$ under the orientation homomorphism, so that 
	\[\Psi_y(h_y)=h_y^{\lambda_y} ,\quad \Psi_y(\bar e) = (\bar e^{\sigma_{\bar e} \mu_y}h_y^{\sigma_{\bar e}\bar \rho_e})^{g_{\bar e}}\]
	Define $\Omega$ on $\hat G_y$ by $\Omega(g) = \Psi_y(g)^{g_{\bar e}^{-1}g_e f}$. Note that the conjugating element $f_{m+1}=g_{\bar e}^{-1}g_e f$ is of the form specified above. We must now check that this map is well-defined- that is, we must check that the relations in $\pi_1 M$ given by this edge are mapped to the trivial element of $\hat H = \widehat{\pi_1 N}$ under $\Omega$. This calculation is essentially the same as the matrix calculations leading to equation \eqref{matreqn}. To reassure the reader that all the signs check out, and to atone for my sins, I will give the computations anyway. The reader should note that throughout we will be heavily using the fact that, for example, every conjugate of $e$ commutes with $h_x$. We will drop $e$ from much of the notation, adding bars when necessary (e.g.\ $\sigma_{\bar e}$ will be written $\bar{\sigma}$).
	\begin{eqnarray*}
		\Omega(h_x^{-1}h^\alpha_y {\bar e}^\gamma) & = & \Psi_x(h_x^{-1})^{f_j}\Psi_y(h^\alpha_y {\bar e}^\gamma)^{g_{\bar e}^{-1} g_e f_j}\\
		& \,\;\quad\sim_{f^{-1}}&  h_x^{-\lambda_x}\left( h_y^{\alpha\lambda_y}({\bar e}^{\gamma\bar{\sigma} \mu_y} h_y^{\bar{\sigma}\bar{\rho}\gamma})^{g_{\bar e}} \right)^{g_{\bar e}^{-1} g_e}\\  
		& \,\;\quad\sim_{\smash g_e^{-1}} & h_x^{-\sigma\lambda_x}h_y^{\bar{\sigma}\alpha\lambda_y+ {\bar{\sigma}\bar{\rho}\gamma}}{\bar e}^{\gamma\bar{\sigma} \mu_y}  \\ 
		& = & (h^{\alpha'}_y {\bar e}^{\gamma'})^{-\sigma\lambda_x}h_y^{\bar{\sigma}\alpha\lambda_y+ {\bar{\sigma}\bar{\rho}\gamma}}{\bar e}^{\gamma\bar{\sigma} \mu_y}  \\
		& = & h_y{}^\wedge\{-\sigma\alpha'\lambda_x + \bar{\sigma}\alpha\lambda_y + \bar\sigma\bar\rho \gamma  \}\bar{e}{}^\wedge\{-\sigma\lambda_x \gamma' + \bar{\sigma}\mu_y \gamma \}\\
		& = & 1
	\end{eqnarray*}
	where $\sim_g$ denotes conjugation by $g$. In the last line we use the definitions \eqref{rhodef}, \eqref{rhobardef} of $\rho$ and $\bar{\rho}$ and the equalities $\sigma\bar{\sigma}=\pm_e 1$ and $\lambda_x=\mu_y$. Secondly, we  have
	\begin{eqnarray*}
	\Omega(e^{-1}h^\beta_y {\bar e}^\delta) & = & \Psi_x(e^{-1})^{f_j}\Psi_y(h^\beta_y {\bar e}^\delta)^{g_{\bar e}^{-1} g_e f_j}\\
	& \,\;\quad\sim_{f^{-1}} & \left( e^{-\sigma\mu_x}h_x^{-\sigma\rho} \right)^{g_e} \left(h_y^{\beta\lambda_y}({\bar e}^{\delta\bar{\sigma} \mu_y} h_y^{\bar{\sigma}\bar{\rho}\delta})^{g_{\bar e}} \right)^{g_{\bar e}^{-1} g_e}\\ 
	& \,\;\quad\sim_{g_e^{-1}} & \left(h_y^{\beta'}{\bar e}^{\delta'} \right)^{-\sigma\mu_x}\left(h_y^{\alpha'}{\bar e}^{\gamma'} \right)^{-\sigma\rho} h_y^{\bar{\sigma}\beta\lambda_y+ {\bar{\sigma}\bar{\rho}\delta}}{\bar e}^{\delta\bar{\sigma} \mu_y}  \\
	&=& h_y{}^\wedge\{-\beta'\sigma\mu_x-\alpha'\sigma\rho + \bar{\sigma}\beta\lambda_y + \bar{\sigma}\bar{\rho}\delta \}\cdot \\
	&&\quad\bar{e}{}^\wedge \{-\sigma\gamma'\rho-\sigma\mu_x\delta' + \delta\bar\sigma \mu_y\}
	\end{eqnarray*}
	Now the exponent of $\bar e$ vanishes by \eqref{rhodef}. To deal with the exponent of $h_y$, recall that as $M$ and $N$ are orientable, maps on boundary tori have determinant $-1$. Hence $\beta\gamma = 1 +\alpha\delta$, and $\beta'\gamma' = 1 +\alpha'\delta'$. Since $-\sigma\gamma'$ is not a zero-divisor in $\hat{\Z}$, it suffices to check that the exponent vanishes when multiplied by $-\sigma\gamma'$.
	\begin{eqnarray*}
	 -\sigma\gamma'(\text{Exponent of }h_y)&=&\beta'\gamma'\mu_x + \alpha'\gamma'\rho - \beta\gamma'\sigma\bar\sigma \lambda_y - \delta\gamma'\sigma\bar\sigma \bar\rho \\
	  &=& (1+\alpha'\delta')\mu_x - (1+\alpha\delta)\lambda_y + \alpha'\gamma'\rho - \delta\gamma\bar{\rho}\\
	  &=& \pm_e\alpha'(\lambda_x\delta-\rho\gamma)-\alpha\delta\lambda_y \pm_e\alpha'\gamma\rho - \delta\gamma\bar{\rho}\\
	  &=& \pm_e\delta\left(\lambda_x\alpha' \pm_e(-\alpha\lambda_y-\gamma\bar{\rho}) \right)\\& =& 0
	\end{eqnarray*}
	where we have freely used \eqref{rhodef} and \eqref{rhobardef}, as well as $\mu_x=\lambda_y$. Thus all relations are satisfied and we have defined $\Omega$ on the tree of groups $(T, \hat G_\bullet)$.
	
	It remains to define $\Omega$ on the stable letters $t_e$ for the HNN extensions over remaining edges. Let $e$ be such an edge. By construction there are elements $g_e\in H_x, g_{\bar e}\in H_y$ and elements $f_x, f_y$ in the subgroup generated by the vertex groups such that 
	\[\Omega|_{G_e}= (\Psi_x)^{f_j}, \quad \Psi_x(e) = (e^{\sigma_e \mu_x}h_x^{\sigma_e\rho})^{g_e}\]
	and similarly for $y$. Then set \[ \Omega(t_e) = f_x^{-1}g_e^{-1} t_e g_{\bar e} f_y\]
	We must of course check that the relations on the edge $e$ are satisfied. The verification of this is, up to conjugacy, almost identical with the verifications above and we will not punish the reader by writing it a second time.
	
	We now have a well-defined homomorphism $\Omega\colon \hat G\to \hat H$. We may check that is an isomorphism by a suitable application of the universal property of a graph of profinite groups. However, since $\Omega$ is not quite a morphism of graphs of groups on the nose- having various conjugating elements involved- we instead give another argument. Recall that $\hat H$ is generated by its vertex groups and the stable letters $t_e$. By induction, each vertex group $\hat H_{x_j}$ lies in the image of $\Omega$, since by construction the subgroup $\hat H_{x_j}^{g_e f_j}$ lies in this image- where $g_e\in \hat H_j$ and where $f_j$ is an element of $\hat H$ of the form 
	\[f_j = k_j k_{j-1}\cdots k_1 \] 
	where $k_i\in \hat H_i$ for $i\leq j$. By induction all the $k_i\, (i<j)$ are in the image of $\Omega$, hence so is 
	\[\hat H_{x_j}^{g_e f_j}= \hat H_{x_j}^{g_e k_j} = \hat H_{x_j} \] 
	Finally, since we have $\Omega(t_e) = l' t_e l$ where $l, l'$ are in the subgroup generated by the $\hat H_i$, it follows that each stable letter $t_e$ is in the image of $\Omega$. So $\Omega$ is surjective.
	
	Finally note that by symmetry we may also construct a surjective homomorphism $\Omega'\colon \hat H\to \hat G$. Both groups being Hopfian, this forces both $\Omega$ and $\Omega'$ to be isomorphisms. The proof is complete.
\end{proof}

\bibliographystyle{alpha}
\bibliography{graphmflds}

\begin{thebibliography}{DFPR82}

\bibitem[AF13]{AF13}
Matthias Aschenbrenner and Stefan Friedl.
\newblock {\em 3-manifold groups are virtually residually p}.
\newblock American Mathematical Society, 2013.

\bibitem[BCR16]{BCR14}
Martin~R. Bridson, Marston~DE Conder, and Alan~W. Reid.
\newblock Determining {F}uchsian groups by their finite quotients.
\newblock {\em Israel Journal of Mathematics}, 214(1):1--41, 2016.

\bibitem[Bel80]{belyui80}
G.~V. Bely\u\i.
\newblock On {G}alois extensions of a maximal cyclotomic field.
\newblock {\em Mathematics of the USSR-Izvestiya}, 14(2):247, 1980.

\bibitem[BF15]{BF15}
Michel Boileau and Stefan Friedl.
\newblock The profinite completion of $3 $-manifold groups, fiberedness and the
  {T}hurston norm.
\newblock {\em arXiv preprint arXiv:1505.07799}, 2015.

\bibitem[BKS87]{BKS87}
Robert~G. Burns, Abe Karrass, and Donald Solitar.
\newblock A note on groups with separable finitely generated subgroups.
\newblock {\em Bulletin of the Australian Mathematical Society},
  36(01):153--160, 1987.

\bibitem[BP12]{BP12}
Riccardo Benedetti and Carlo Petronio.
\newblock {\em Lectures on hyperbolic geometry}.
\newblock Springer Science \& Business Media, 2012.

\bibitem[BR15]{BR15}
Martin~R. Bridson and Alan~W. Reid.
\newblock Profinite rigidity, fibering, and the figure-eight knot.
\newblock {\em arXiv preprint arXiv:1505.07886}, 2015.

\bibitem[Bri07]{brinnotes}
Matthew~G. Brin.
\newblock Seifert fibered spaces: Notes for a course given in the spring of
  1993.
\newblock {\em arXiv preprint arXiv:0711.1346}, 2007.

\bibitem[BRW16]{BRW16}
Martin~R. Bridson, Alan~W. Reid, and Henry Wilton.
\newblock Profinite rigidity and surface bundles over the circle.
\newblock {\em arXiv preprint arXiv:1610.02410}, 2016.

\bibitem[DFPR82]{DFPR}
John~D. Dixon, Edward~W. Formanek, John~C. Poland, and Luis Ribes.
\newblock Profinite completions and isomorphic finite quotients.
\newblock {\em Journal of Pure and Applied Algebra}, 23(3):227--231, 1982.

\bibitem[GDJZ15]{GDJZ15}
Gabino Gonz{\'{a}}lez-Diez and Andrei Jaikin-Zapirain.
\newblock The absolute {G}alois group acts faithfully on regular dessins and on
  {B}eauville surfaces.
\newblock {\em Proceedings of the London Mathematical Society},
  111(4):775--796, 2015.

\bibitem[GR78]{GR78}
Dion Gildenhuys and Luis Ribes.
\newblock Profinite groups and {B}oolean graphs.
\newblock {\em Journal of Pure and Applied Algebra}, 12(1):21--47, 1978.

\bibitem[Gr{\"{u}}57]{Gru57}
Karl~W. Gr{\"{u}}nberg.
\newblock Residual properties of infinite soluble groups.
\newblock {\em Proceedings of the London Mathematical Society}, 3(1):29--62,
  1957.

\bibitem[Ham01]{Ham01}
Emily Hamilton.
\newblock Abelian subgroup separability of {Haken} 3-manifolds and closed
  hyperbolic $n$-orbifolds.
\newblock {\em Proceedings of the London Mathematical Society},
  83(03):626--646, 2001.

\bibitem[Hem14]{hempel14}
John Hempel.
\newblock Some 3-manifold groups with the same finite quotients.
\newblock {\em arXiv preprint arXiv:1409.3509}, 2014.

\bibitem[HWZ12]{HWZ12}
Emily Hamilton, Henry Wilton, and Pavel~A. Zalesskii.
\newblock Separability of double cosets and conjugacy classes in 3-manifold
  groups.
\newblock {\em Journal of the London Mathematical Society}, pages --040, 2012.

\bibitem[HZ12]{HeZ12}
W.~N. Herfort and Pavel~A. Zalesskii.
\newblock Addendum: Virtually free pro-$ p $ groups whose torsion elements have
  finite centralizer.
\newblock {\em arXiv preprint arXiv:0712.4244}, 2012.

\bibitem[NS07]{NS07}
Nikolay Nikolov and Dan Segal.
\newblock On finitely generated profinite groups, {I}: strong completeness and
  uniform bounds.
\newblock {\em Annals of mathematics}, pages 171--238, 2007.

\bibitem[NW01]{niblowise}
Graham Niblo and Daniel Wise.
\newblock Subgroup separability, knot groups and graph manifolds.
\newblock {\em Proceedings of the American Mathematical Society},
  129(3):685--693, 2001.

\bibitem[RZ96]{RZ96}
Luis Ribes and Pavel~A. Zalesskii.
\newblock Conjugacy separability of amalgamated free products of groups.
\newblock {\em Journal of Algebra}, 179(3):751--774, 1996.

\bibitem[RZ00a]{RZ00p}
Luis Ribes and Pavel Zalesskii.
\newblock Pro-$p$ trees and applications.
\newblock In Marcus du~Sautoy, Dan Segal, and Aner Shalev, editors, {\em New
  Horizons in pro-p Groups, Progess in Mathematics volume 184}, pages 75--119.
  Birkh{\"{a}}user Boston, 2000.

\bibitem[RZ00b]{RZ00}
Luis Ribes and Pavel Zalesskii.
\newblock {\em Profinite groups}.
\newblock Springer, 2000.

\bibitem[RZ01]{RZup}
L.~Ribes and P.~Zalesskii.
\newblock {\em Profinite Trees}.
\newblock Unpublished book, 2001.

\bibitem[Sco83]{scott83}
Peter Scott.
\newblock The geometries of 3-manifolds.
\newblock {\em Bulletin of the London Mathematical Society}, 15(5):401--487,
  1983.

\bibitem[Ser03]{Serre}
Jean-Pierre Serre.
\newblock Trees. {T}ranslated from the {F}rench original by {J}ohn {S}tillwell.
  {C}orrected 2nd printing of the 1980 {E}nglish translation.
\newblock {\em Springer Monographs in Mathematics. Springer-Verlag, Berlin},
  2003.

\bibitem[Wil15]{Wilkes15}
Gareth Wilkes.
\newblock Profinite rigidity for {S}eifert fibre spaces.
\newblock {\em Geometriae Dedicata}, pages 1--23, 2015.

\bibitem[WZ10]{WZ10}
Henry Wilton and Pavel Zalesskii.
\newblock Profinite properties of graph manifolds.
\newblock {\em Geometriae Dedicata}, 147(1):29--45, 2010.

\bibitem[WZ17]{WZ14}
Henry Wilton and Pavel Zalesskii.
\newblock Distinguishing geometries using finite quotients.
\newblock {\em Geom. Topol.}, 21(1):345--384, 2017.

\bibitem[Zal89]{Zal89}
P.~A. Zalesskii.
\newblock A geometric characterization of free formations of profinite groups.
\newblock {\em Siberian Mathematical Journal}, 30(2):227--235, 1989.

\bibitem[ZM89a]{ZM89a}
PA~Zalesskii and Oleg~Vladimirovich Mel'nikov.
\newblock Subgroups of profinite groups acting on trees.
\newblock {\em Mathematics of the USSR-Sbornik}, 63(2):405, 1989.

\bibitem[ZM89b]{ZM89b}
Pavel~Aleksandrovich Zalesskii and Oleg~Vladimirovich Mel{'}nikov.
\newblock Fundamental groups of graphs of profinite groups.
\newblock {\em Algebra i Analiz}, 1(4):117--135, 1989.

\end{thebibliography}
\end{document}